\documentclass[12pt]{amsart}		    
\usepackage{amsmath,amssymb,amsthm,amsxtra}
\newtheorem{theorem}{Theorem}
\newtheorem{lemma}[theorem]{Lemma}
\newtheorem{proposition}[theorem]{Proposition}
\newtheorem{corollary}[theorem]{Corollary}
\newtheorem*{mthm}{Main Theorem}
\newtheorem*{mcor}{Corollary}
\newtheorem*{definition}{Definition}

\theoremstyle{remark}
\newtheorem*{Rem}{Remark}
\newtheorem*{Rems}{Remarks}
\newtheorem*{FRems}{Final Remarks}

\newcommand{\Proof}{\begin{proof}}

\newcommand{\card}[1]{\lvert#1\rvert}
\newcommand{\abs}[1]{\lvert#1\rvert}
\newcommand{\norm}[1]{\lVert#1\rVert}
\newcommand{\iin}{\!\in\!}
\newcommand{\complex}{\mathbb{C}}
\newcommand{\Rgroup}{\mathbb{R}}
\newcommand{\Tgroup}{\mathbb{T}}
\newcommand{\Nnat}{\mathbb{N}}

\newcommand{\Meas}{\mathbf{M}}
\newcommand{\sMeas}{\mathbf{M_s}}
\newcommand{\aMeas}{\mathbf{M_a}}
\newcommand{\ssMeas}{\mathbf{M_{ss}}}
\newcommand{\aiMeas}{\mathbf{M_{ai}}}
\newcommand{\MeasG}{\Meas(G)}
\newcommand{\sMeasG}{\sMeas(G)}
\newcommand{\aMeasG}{\aMeas(G)}

\newcommand{\ssMeast}[1]{\Meas_{\mathbf{ss},#1}}
\newcommand{\aiMeast}[1]{\Meas_{\mathbf{ai},#1}}
\newcommand{\unitball}{\mathbf{B_1}}

\newcommand{\scrC}{\mathcal{C}}
\newcommand{\scrD}{\mathcal{D}}
\newcommand{\scrO}{\mathcal{O}}

\newcommand{\scrK}{\mathcal{K}}
\newcommand{\ssubgr}{\mathcal{K}_{\chi(G)}^\circ}
\newcommand{\wstar}{weak$^\ast$\,}
\newcommand{\conv}{\star}
\newcommand{\cl}[1]{\overline{#1}}
\DeclareMathOperator{\lspan}{span}
\newcommand{\squ}{\mathbin{\square}}

\date{October 8, 2015}

\begin{document}
\title{Proof of the Ghahramani--Lau conjecture}
\author{Viktor Losert}
\address{Institut f\"ur Mathematik, Universit\"at Wien\\ Strudlhofg.\ 4\\
  A 1090 Wien, Austria}
\email{viktor.losert@univie.ac.at}
\author{Matthias Neufang}
\address{School of Mathematics and Statistics, Carleton University,
1125 Colonel By Dr., \\
Ottawa, Ontario K1S 5B6, Canada}
\address{Laboratoire de Math\'{e}matiques Paul Painlev\'{e}
(UMR CNRS 8524),
Universit\'{e} Lille 1 -- Sciences et Technologies, UFR de Math\'{e}matiques,
    59655 Villeneuve d'Ascq Cedex, France}
\email{mneufang@math.carleton.ca, matthias.neufang@math.univ-lille1.fr}
\author{Jan Pachl}
\address{The Fields Institute for Research in Mathematical Sciences,
222 College St., \\
Toronto, Ontario M5T 3J1, Canada}
\thanks{Matthias Neufang and Juris Stepr\={a}ns were partially supported
by NSERC; this support is greatfully acknowledged.
Jan Pachl appreciates the opportunity of working in the supportive environment
at the Fields Institute.}
\author{Juris Stepr\={a}ns}
\address{Department of Mathematics and Statistics, York University,
4700 Keele St., \\
Toronto, Ontario M3J 1P3, Canada}
\email{steprans@yorku.ca}

\begin{abstract}
\noindent
The Ghahramani--Lau conjecture is established; in other words,
the  measure algebra of every locally compact group is strongly Arens
irregular. To this end, we introduce and study certain new classes of measures
(called approximately invariant, respectively, strongly singular) which
are of interest in their own right. Moreover, we show that the same result
holds for the measure algebra of any (not necessarily locally compact)
Polish group.
\end{abstract}
\maketitle
\renewcommand{\baselinestretch}{1.2}\normalsize
\section{Introduction}
    \label{sec:intro}

The study of the Arens products on the second dual of a Banach algebra has been an
active area of functional analysis and abstract harmonic analysis for many years.
As is well-known, operator algebras (and their quotient algebras) are Arens
regular, i.e., both Arens products on the second dual coincide. However, the
situation is radically different for group algebras such as the convolution
algebra $L_1(G)$ over a locally compact group $G$: building on the pioneering
work (in the abelian case) of Civin--Yood~\cite{CiYo} from 1961, N.J. Young
\cite{Yo} showed in 1973 that $L_1(G)$ is never Arens regular unless $G$ is
finite. It was thus natural to ask how irregular the multiplication in the
bidual of $L_1(G)$ is -- this was only settled in 1988 by
Lau--Losert~\cite{LauLosert}
who showed that $L_1(G)$ is strongly Arens irregular
(in the terminology established in~\cite{DaLa}); in other words, left
multiplication by $m \in L_1(G)^{**}$ on $L_1(G)^{**}$ is the same with
respect to both Arens products only if $m \in L_1(G)$, and this holds
as well for right multiplication by $m$.

Since the measure algebra $\MeasG$ contains $L_1(G)$ as a closed subalgebra
(in fact, as an ideal), it is clear by Young's theorem that $\MeasG$ is only
Arens regular for finite groups $G$. It was conjectured by
Lau~\cite{Lau1994ffs} and
Ghahramani--Lau~\cite{GhLa} that, as in the case of $L_1(G)$, the measure
algebra $\MeasG$ is also strongly Arens irregular for any locally compact
group $G$. In \cite{Neu}, this was established for two classes of locally
compact non-compact groups: those whose cardinality is a non-measurable
cardinal, and those for which the relation $\kappa(G) \geq 2^{\chi(G)}$ holds,
where $\kappa(G)$ denotes the compact covering number, and $\chi(G)$ the
local weight, also known as the character.

In this paper, we prove the Ghahramani-Lau conjecture for all locally compact
groups.

\section{Main Result and some basics}\vspace{1mm}
    \label{sec:notation}

Vector spaces are over the reals $\Rgroup$ or over the complex field
$\complex$.
The linear span of a set $D$ in a vector space is \,$\lspan(D)$.
When $M$ is a normed space, $\unitball (M)$ is the unit ball in $M$,
and $M^\ast$ is the dual of $M$.
For $f\iin M^\ast,\;\mu\iin M$,
the value of the functional $f$ at $\mu$ will be written as
$\langle f,\mu\rangle$.
As usual, $M$ is identified with a subspace of the second dual
$M^{\ast\ast}$ by the canonical embedding (which amounts to
$\langle \mu,f\rangle=\langle f,\mu\rangle$).

Now assume that $M$ is a Banach algebra. The multiplication $\conv$ of $M$
is extended to the \emph{left} (or first) \emph{Arens product} $\squ$ on
$M^{\ast\ast}$ as follows,
first defining a right action $\cdot$ of $M$ on $M^\ast$ and then a left
action $\odot$ of $M^{\ast\ast}$ on $M^\ast$\,:
\begin{eqnarray*}
\langle h\cdot\mu,\,\nu\,\rangle & = & \langle\,h\,,\mu\conv\nu\rangle
	\text{\;\;for\;\;} h\iin M^\ast,\;\mu,\nu\iin M\,,\\
\langle\mathfrak{n}\odot h,\,\mu\,\rangle & =
& \langle\,\mathfrak{n}\,,h\cdot\mu\rangle \text{\;\;\;for\;\;}
    \mathfrak{n}\iin M^{\ast\ast},\; h\iin M^\ast,\; \mu\iin M\,,\\
\langle\mathfrak{m}\squ\mathfrak{n},\,h\,\rangle & = &
    \langle\,\mathfrak{m}\, , \mathfrak{n} \odot h \rangle \text{\;\;for\;\;}
    \mathfrak{m},\mathfrak{n}\iin M^{\ast\ast},\; h\iin M^\ast .
\end{eqnarray*}
Note that if $\mu,\nu\iin M$ and $h\iin M^\ast$
then (using the embedding of $M$ into $M^{\ast\ast}$)
\[\tag{\#}\langle\mu\squ\nu,\,h\,\rangle=
\langle\nu\odot h,\,\mu\,\rangle =
\langle h\cdot\mu,\,\nu\,\rangle  =
\langle\, h\,,\mu\conv\nu\rangle\,.\]

\noindent The (left) {\it topological centre} of $M^{\ast\ast}$ is defined as
\[
Z_t(M^{\ast\ast})
= \{\; \mathfrak{m} \iin M^{\ast\ast}  : \;
\text{the\;mapping\ \;} \mathfrak n\mapsto\mathfrak m\squ\mathfrak n
\text{\ is\;\wstar continuous\;on\;} M^{\ast\ast} \,\}.
\]
It is easy to show that $M\subseteq Z_t(M^{\ast\ast})$ \;(see
\cite{Dales}\;p.\,248ff. for further discussion of Arens products and
topological centres). There is a second canonical method to extend the
multiplication of $M$\,. It gives the right (or second) Arens product. The
left topological
centre consists of those elements $\mathfrak m\iin M^{\ast\ast}$ for which
the two Arens products coincide whenever $\mathfrak{m}$ is the left factor
(\cite{Dales}\;Def.\,2.6.19). In \cite{DaLa}\; refined notions have been
introduced. There, $Z_t$ is denoted as $Z_t^{(1)}$ and one has a
corresponding notion of right topological centre $Z_t^{(2)}$. Then  (as the
minimal case) $M$ is called {\it strongly Arens irregular} if both topological
centres coincide with $M$ (\cite{DaLa}\;Def.\,2.18).\vspace{1mm}

When $\Omega$ is a locally compact topological space,
$\Meas(\Omega)$ is the Banach space of bounded real or complex Radon measures
on $\Omega$ with the total variation norm. $C_0(\Omega)$ denotes the space
of real- or complex-valued continuous functions on $\Omega$ vanishing
at infinity, equipped with supremum norm. $\Meas(\Omega)$ is identified with
$C_0(\Omega)^\ast$ by \,$\langle \mu,f\rangle=\int\!fd\mu$\,
(Riesz Representation Theorem). When $\Omega=G$
is a locally compact topological group, $\MeasG$ is a Banach algebra
with convolution $\conv$ \;(\cite{Dales}\;p.\,374ff.).\vspace{1mm}

\begin{mthm}   
$Z_t\bigl(\MeasG^{\ast\ast}\bigr)=\MeasG$ holds for every locally compact
group $G$\,.
\end{mthm}

This was the conjecture of Ghahramani-Lau.
It was proved by Lau~\cite{Lau1986cam}
for discrete groups, and by Neufang~\cite{Neu} when
$\kappa(G)\geq 2^{\chi(G)}$ (and in some other cases).

\begin{mcor}	  
$\MeasG$ is always strongly Arens irregular.
\end{mcor}
\Proof
The corresponding result for the right topological centre follows by  using
that $Z_t^{(2)}(M^{**}) = Z_t^{(1)}((M^{\text{op}})^{**})$
(\cite{DaLa}\;p.\,22),
where $M^\text{op}$ denotes the opposite algebra of~$M$. Furthermore,
$\MeasG^\text{op} = \Meas(G^\text{op})$.
\end{proof}
\vspace{1mm plus 1mm}

If $K$ is a closed subgroup (not necessarily normal) of a topological group
$G$, then $G/K$ denotes the space of left cosets with the
quotient topology (\cite{Hewitt1963aha}\;Def.\,5.15).

For set-theoretic notions, e.g. definitions about cardinal numbers,
see \cite{Je} or \cite{Ku} (see also Sections \ref{sec:singularSeparation} and
\ref{sec:strsingularSeparation} for comments about ordinal numbers).
If $\Omega$ is a topological space, the {\it compact covering number} of
$\Omega$, denoted $\kappa(\Omega)$, is the least
cardinal $\tau$ such that $\Omega$ is a union of $\tau$ compact subsets.
$d(\Omega)$, the {\it density character} (shortly: density) is the least
cardinal of a dense subset of $\Omega$\,.
For $\omega\iin\Omega$ define $\chi(\Omega)$ \,({\it local weight} or
character) to be the least cardinal~$\tau$ such that $\omega$ has a base
of neighbourhoods of cardinality~$\tau$ \;(more accurately, one should
write $\chi(\omega,\Omega)$, but if the group of homeomorphisms
acts transitively on $\Omega$
it does not depend on the choice of $\omega$\,;
in particular for $\Omega=G$ or $\Omega=G/K$ which are the cases used below).
When $K$ is compact, the coset space $G/K$ is metrizable
if and only if \,$\chi(G/K)\leq\aleph_0$ \,
(the group case is well known: \cite[8.3]{Hewitt1963aha};
the general case was done by Kristensen: \cite[8.14d]{Hewitt1963aha}\,).

The cardinality of a set $E$ is denoted by $\card E$.
For any locally compact group $G$ and a compact subgroup $K$ of infinite
index, we have
\,$\card{G/K}=\kappa(G)\cdot 2^{\chi(G/K)}$ \,(this can be shown as in
the group case, compare \cite{HuN}\;L.\,5.4). Furthermore, we have
\,$d\bigl(\Meas(G/K)\bigr)=\card{G/K}$
\;(this is trivial when $K$ is open, otherwise it follows from
\cite{HuN}\;Thm.\,5.5).

When $\mu,\nu\iin\Meas(\Omega)$, $\mu\ll\nu$ means that $\mu$ is absolutely
continuous with respect to $\nu$ and
$\mu\perp\nu$ means that $\mu$ and $\nu$ are mutually
singular.
We have $\mu\perp\nu$ if and only if $\abs{\mu}\perp\abs{\nu}$.
When $D,D' \subseteq\Meas(\Omega)$, $D\perp D'$ means that $\mu\perp\nu$
whenever $\mu\iin D$ and $\nu\iin D'$.
If $\Omega'$ is a
locally compact subspace of $\Omega$\,, any measure on $\Omega'$ has a
trivial extension to $\Omega$ (defining it to be zero outside $\Omega'$).
In this way, $\Meas(\Omega')$ will be considered as a subspace of
$\Meas(\Omega)$  \;(if $\Omega'$ is closed, this embedding is just the dual
of the restriction
operator $C_0(\Omega)\to C_0(\Omega')$\,). In particular, if $H$ is a closed
subgroup of $G$, we consider $\Meas(H)$ as a (\wstar closed) subalgebra of
$\MeasG$\,.

$\aMeasG$ denotes the subspace of those measures in $\MeasG$ that
are absolutely continuous
with respect to some (left or right) Haar measure
$\lambda_G$ ($G$ a locally compact group). Fixing $\lambda_G$\,, this will be
identified in the usual way with the space $L^1(G)$ defined by $\lambda_G$\,.
$\sMeasG$ is the subspace of the measures $\mu\iin\MeasG$
such that $\mu\perp\lambda_G$\,;
equivalently, $\mu\perp\nu$ for every $\nu\iin\aMeasG$ \,
(this notation differs from that in~\cite{Hewitt1963aha}\,(19.13),
where $\sMeasG$ stands for the space of singular \emph{continuous} measures).
$\delta_x$ ($\in\!\MeasG$\,) denotes the point mass at $x\iin G$\,.

If $\nu\iin \MeasG$ and $h\iin C_0(G)$,
then (using the embedding of $C_0(G)$ into $\MeasG^\ast=C_0(G)^{\ast\ast}$)
we have \,$\nu\odot h\in C_0(G)$, \,defining a left action of $\MeasG$ on
$C_0(G)$ \,(\cite{Dales}\linebreak p.\,376ff., where this is denoted as
\,$\nu\cdot h$\,). By \thetag{\#},
\,$\langle\nu\odot h,\,\mu\,\rangle =\langle\, h\,,\mu\conv\nu\rangle$\,,
in particular
\,$\delta_x\!\odot h\,(y)=h(yx)$ \,for $x,y\in G$ \,(right translation of
functions).

If $K$ is a compact subgroup of $G$\,, $\lambda_K\iin\Meas(K)\subseteq\MeasG$
shall always be its normalized Haar measure (i.e., $\lambda_K(K)=1$).
Recall that $\lambda_K\star\lambda_K=\lambda_K$ and for those readers who are
more acquainted to classical convolution, we mention that (using invariance
of $\lambda_K$ under the group inversion) one has
$\lambda_K\odot h=h\star\lambda_K$ \,(but, in principle, all our results
about $\odot$ can be proved using duality arguments). If $K$ is normal in
$G$\,, then $\lambda_K$ is central in $\MeasG$ (and conversely).

\begin{lemma}\label{Identify}
For $G$ a locally compact group and a compact subgroup $K$ of $G$
put $\Meas(G,K)=\{\mu\in\MeasG :\,\mu=\mu\star\lambda_K\}=
\MeasG\star \lambda_K$\,.
Then the canonical mapping $\pi\!:G\to G/K$ induces an
isometric isomorphism between $\Meas(G,K)$ and $\Meas(G/K)$.
\end{lemma}\noindent
Before giving the proof, we state some conventions that will be important
throughout the paper.
Using this isomorphism as identification, $\Meas(G/K)$ will be considered as a
\wstar closed
subspace (left ideal) of $\MeasG$ (this corresponds to the attitude of
\cite[22.6]{Dieud}). Note that if $x\iin G$ normalizes $K$
(i.e., $xKx^{-1}=K$), we have\linebreak $\Meas(G/K)\conv \delta_x=\Meas(G/K)$.
If $K$ is normal, the multiplication on $\Meas(G/K)$ inherited from $\MeasG$
corresponds to
convolution defined by the quotient group $G/K$\,. Similarly, $C_0(G/K)$
will be identified with $\lambda_K\odot C_0(G)$
\,(right $K$-periodic functions in $C_0(G)$\,). One can use this also to define
the subspaces $\aMeas(G,K)=\aMeasG\cap\Meas(G/K)$ and
$\sMeas(G,K)=\sMeasG\cap\Meas(G/K)$. Alternatively, there exists a left
$G$-invariant measure $\lambda_{G/K}$ on $G/K$\, (e.g. by
\cite[Thm.\,15.24]{Hewitt1963aha}) and
$\aMeas(G,K)$ (resp. $\sMeas(G,K)$\,) consist of the measures
$\mu\in\Meas(G/K)$ with $\mu\ll\lambda_{G/K}$ (resp. $\mu\perp\lambda_{G/K}$).

\Proof
This is contained in the proof of Thm.\,8.1B of \cite{Jew} (in a different
notation). We include a direct argument. $f\mapsto f\circ\pi$ defines
an isometric linear mapping $C_0(G/K)\to C_0(G)$\,. Its image consists of
all right $K$-periodic functions in $C_0(G)$ \,(i.e., $h=f\circ\pi$ satisfies
$h(yx)=h(y)$ for all
$x\iin K,\,y\iin G$). For general $h\iin C_0(G)$ we have
$\lambda_K\odot h(y)=\langle\lambda_K\odot h,\,\delta_y\,\rangle =
\int\! h(yx)\,d\lambda_K(x)$. It follows easily that right periodicity of $h$
is equivalent to \,$h=\lambda_K\odot h$\,. Since $\lambda_K$ is idempotent,
we can see that the image of the embedding coincides with
$\lambda_K\odot C_0(G)$.

For $\mu\in\MeasG$ the {\it image measure} $\dot\mu\in\Meas(G/K)$ is defined
by
$\dot\mu(B)=\mu(\pi^{-1}(B))$. Easy computations show that $\mu\mapsto\dot\mu$
is just the dual mapping of the embedding of $C_0$ described above and that
$(\mu\star\lambda_K)^{\textstyle\cdot}=\dot\mu$\,. Finally, one can conclude
(from the properties
at the $C_0$-level) that the restriction to $\Meas(G,K)$ is isometric and
surjective. Alternatively: \ the inverse mapping from $\Meas(G/K)$ to
$\Meas(G,K)$ is a special case of the mapping $m\mapsto m^\sharp$
investigated in \cite[Chap.\,7, \S\,2]{Bourb} \,(the dual mapping of
\cite[Thm.\,15.21]{Hewitt1963aha}).
\end{proof}\noindent
The following notion will provide an essential tool for our argument.
\begin{definition}		   
Let $\tau$ be a cardinal.
Say that a measure $\mu\iin\MeasG$ is \emph{$\tau$-thin}
if there is a set $P\subseteq G$
such that \,$\card{P}=\tau$ and
\,$\mu\conv\delta_p\perp\mu\conv\delta_{p'}$ for all
$p,p'\iin P$ with $p\neq p'$.
\end{definition}

The basic observation used in \cite{Neu} was that $Z_t(M^{\ast\ast})$ must
be small whenever there exists $h\iin M^\ast$ such that
$\unitball (M^{\ast\ast})\odot h$ is big. A general version is worked out
in Section \ref{sec:dirsums}, see the
Lemmas \ref{lemma:centresubsp} and \ref{lemma:centresubsp2} for precise
conditions. This motivates the search for factorization theorems. In
the case of $\MeasG$ it has turned out to be very difficult to get
information about the behaviour of $\mathfrak n\odot h$ for general
$\mathfrak n\iin M^{\ast\ast}$. Thus we have restricted to
$\mathfrak n=\delta_x$
($x\iin G$) and limits $\mathfrak n\in\cl{\delta(G)}$\; (we write
\,$\delta(G) = \{ \delta_x :\, x\iin G \}$, \linebreak
$\overline{\phantom{tt}}$ \,denotes
the \wstar closure in the bidual). For this case, factorizations are
constructed in Section~\ref{sec:factorization}, using
thinness of measures (Theorems \ref{th:factorization} and
\ref{th:factorizationsub}).

It is almost immediate that if $\kappa(G)$ is uncountable, any measure
$\mu\iin\MeasG$ is\linebreak
$\kappa(G)$\,-thin (recall that the support of $\mu$ is always contained in
a $\sigma$-compact subgroup) and this was used in \cite{Neu} to prove the
conjecture in the case where $\kappa(G)\geq 2^{\chi(G)}$.
But this approach is insufficient when $\kappa(G)$ is small (in particular
for compact~$G$). In Section \ref{sec:singularSeparation}, we consider the
case of singular measures, where stronger conclusions are possible. Extending
a technique of \cite{Prokaj2003csm} to general locally compact groups,
we will show (Theorem \ref{th:perfectset}) that if
$G$ is non-discrete, any
$\mu\iin\sMeasG$ is $2^{\aleph_0}\kappa(G)$\,-thin. Using this, we can prove
in Section \ref{sec:caseI}
the Main Theorem for metrizable groups (and also for all $G$ with
$\card{G}\leq 2^{\aleph_0}\kappa(G)$\,), see Theorem \ref{th:caseI}.
To make this important case more easily accessible, we provide some ``light"
versions of the preliminary results (Corollaries \ref{cor:centresum} and
\ref{cor:thincentre}).

For the proof of the Main Theorem in the general case,
the basic idea is to consider subspaces $\Meas(G/K)$ \,($K$ a compact
subgroup) and to use induction on $\chi(G/K)$. In Section \ref{sec:compsub}
we collect first some results on the behaviour of $\chi(G/K)$ and
the possibilities to replace $K$ by a normal subgroup. In the non-metrizable
case a further refinement of the decomposition of
$\MeasG$ is introduced (Theorem \ref{th:measureclasses}). Instead of singular
measures we consider ``strongly singular measures". In
Section~\ref{sec:strsingularSeparation} we show that any
\,$\mu\in\ssMeas(G,K)$ is $\card{G/K}$\,-thin (Corollary \ref{cor:thinss}).
Then in Section~\ref{sec:caseII} the final work is done, proving
Theorem \ref{th:caseII} which contains our Main Theorem.
\vspace*{3mm plus 6mm}

\section{Subspaces and Direct sums}\vspace{1mm}
    \label{sec:dirsums}

Let $M$ be a Banach space, $M_1$ a closed subspace.
$M_1^\circ\subseteq M^\ast$
shall denote the annihilator of $M_1$ (the continuous functionals vanishing
on $M_1$).
Recall that $M_1^\circ$ can be identified with the dual space $(M/M_1)^\ast$,
the \wstar topology being just the induced topology from $M^\ast$. The bidual
$M_1^{\ast\ast}$ can be identified with a subspace of $M^{\ast\ast}$ (using
the second dual of the inclusion mapping) and this is compatible with the
embedding
of $M_1$ into its bidual. Then, $M_1^{\ast\ast}$ is just the \wstar closure
of $M_1$ in $M^{\ast\ast}$ and $M_1^{\ast\ast}=M_1^{\circ\circ}$ (where the
second annihilator refers to $M^{\ast\ast}$).

\begin{lemma}
    \label{lemma:wcontinuity}
Let $M$ be a Banach algebra, $M_1$ a closed subspace.
\item[(1)]
If \,$\mu\iin\unitball (M)$ is such that $M_1\conv\mu\subseteq M_1$\,, then \
$\mu \odot \unitball (M_1^\circ) \subseteq \unitball (M_1^\circ)$ \ and \
$\unitball (M_1^{\ast\ast})\squ\mu\subseteq\unitball (M_1^{\ast\ast})$.
\item[(2)]
For every \,$h\iin M^\ast$
the mapping $\mathfrak m\mapsto\mathfrak m\odot h$
from $M^{\ast\ast}$ to $M^\ast$
is continuous for the \wstar topologies on $M^{\ast\ast}$ and $M^\ast$.
\end{lemma}

\Proof
This follows from the definitions of $\odot$ and $\squ$\,.
\end{proof}

\begin{lemma}
    \label{lemma:centresubsp}
Let $M$ be a Banach algebra, $M_1$ a closed subspace, and assume
that there exists \,$h\iin M^\ast$ such that
\,$\unitball (M^{\ast\ast})\odot h\supseteq\unitball (M_1^\circ)$.
Then  \,$Z_t(M^{\ast\ast})\subseteq M+M_1^{\ast\ast}$.
\end{lemma}

\Proof
Take \,$\mathfrak m\in Z_t(M^{\ast\ast})$. The idea is to show that the
assumptions of the lemma imply \wstar continuity of the restriction of
$\mathfrak{m}$ to $M_1^\circ$. Then (recall that the isomorphism of
$M_1^\circ$ and $(M/M_1)^\ast$ respects the \wstar topologies) there exists
$\mu\iin M$ such that $\langle\mathfrak m,u\rangle=\langle u,\mu\rangle$
holds for all $u\iin M_1^\circ$. Consequently, \,$\mathfrak m-\mu\in
M_1^{\circ\circ}=M_1^{\ast\ast}$.

To show \wstar continuity, we consider the functional
\,$\psi\!:\mathfrak n\mapsto \langle\mathfrak m,\mathfrak n\odot h\rangle$
on $M^{\ast\ast}$. \linebreak
$\psi$ is \wstar continuous, since
$\langle\mathfrak m,\mathfrak n\odot h\rangle=
\langle\mathfrak m\squ\mathfrak n,h\rangle$ and
$\mathfrak m\in Z_t(M^{\ast\ast})$.
Let $\tau_1$ be the quotient topology on
$\unitball (M^{\ast\ast})\odot h$ obtained from the \wstar topology
on $\unitball (M^{\ast\ast})$ by the mapping
$\mathfrak{n}\mapsto\mathfrak{n}\odot h$. Weak$^\ast$ continuity of $\psi$
implies that $\mathfrak m$ is
$\tau_1$-continuous. Using compactness and part (2) of
Lemma~\ref{lemma:wcontinuity} one can see that $\tau_1$ coincides with
the topology on $\unitball (M^{\ast\ast})\odot h$ induced by the
\wstar topology of $M^\ast$. By our assumption on $h$, this implies that
\,$\mathfrak m\vert\,\unitball (M_1^\circ)$ is \wstar continuous.
Then, by the Krein-\v Smulian (or Banach-Dieudonn\'e) theorem,
\,$\mathfrak{m}\vert M_1^\circ$ is \wstar continuous. 
\end{proof}

\begin{corollary}
    \label{cor:centresum}
Assume that $M=M_0\oplus M_1$ is the topological direct sum of closed
subspaces and $M_1$ satisfies the assumption of Lemma \ref{lemma:centresubsp}.
Then  \,$Z_t(M^{\ast\ast})\subseteq M_0\oplus M_1^{\ast\ast}$.
\end{corollary}\noindent
We have $M/M_1\cong M_0$ (hence $M_1^\circ\cong M_0^\ast$). Here $\cong$
shall mean that these
are isomorphic Banach spaces, but not necessarily isometric.
\vspace{1mm plus 1mm}

For the proof of our main theorem in the non-metrizable case, we need another
variation of Lemma \ref{lemma:centresubsp}. If $M_2$ is a closed subspace
of $M$, \,$p_2\!:M^\ast\to M^\ast/M_2^\circ$ denotes the canonical projection.
When $M^\ast/M_2^\circ$ is identified with $M_2^\ast$, this is the dual map
of the inclusion $M_2\to M$ (hence \wstar continuous). It assigns to
a functional $h\iin M^\ast$ its restriction $h\vert M_2$ to $M_2$.

\begin{lemma}
    \label{lemma:centresubsp2}
Let $M$ be a Banach algebra, $M_1\!\subseteq M_2$ closed subspaces of $M$ and
assume that there exists \,$h\iin M^\ast$ such that
\,$p_2\bigl(\unitball (M^{\ast\ast})\odot h\bigr)\supseteq
p_2\bigl(\unitball (M_1^\circ)\bigr)$.
Then  $Z_t(M^{\ast\ast})\cap M_2^{\ast\ast}\subseteq M_2+ M_1^{\ast\ast}$\,.
\end{lemma}

\Proof
This is similar to the proof of Lemma \ref{lemma:centresubsp}. Take
$\mathfrak m\in Z_t(M^{\ast\ast})\cap M_2^{\ast\ast}$.
Since $\mathfrak m\iin M_2^{\ast\ast}$, it induces a linear
functional $\mathfrak{m'}$ on $M^\ast/M_2^\circ$, satisfying
$\mathfrak{m}=\mathfrak{m'}\circ p_2$\,. Considering $\psi$ as in
Lemma \ref{lemma:centresubsp} and the quotient topology on
$p_2\bigl(\unitball (M^{\ast\ast})\odot h\bigr)$
arising from the composed mapping
\,$\mathfrak{n}\mapsto\mathfrak{n}\odot h\mapsto p_2(\mathfrak{n}\odot h)$,
the assumption \linebreak
$\mathfrak m\in Z_t(M^{\ast\ast})$ and the condition on $h$ imply as above
\wstar continuity of $\mathfrak{m'}\vert\, p_2(M_1^\circ)$.
Then there exists
$\mu\iin M_2$ such that $\langle\mathfrak{m'},u\rangle=\langle u,\mu\rangle$
holds for all $u\in p_2(M_1^\circ)$ and our claim follows.
\end{proof}

\section{Factorization for thin measures}\vspace{1mm}
    \label{sec:factorization}

The purpose of this section is to prove Theorem~\ref{th:factorization} and
Corollary~\ref{cor:thincentre},
which are key steps in the proof of the main result.
An extended version will be given in Theorem~\ref{th:factorizationsub}.

For $\mu\in\MeasG$ and $x\iin G$
write \,$\mu \conv x = \mu \conv\delta_x$\,.
Thus \,$\mu\conv x(Bx) = \mu(B)$ when $B\subseteq G$ is a Borel set
(or $|\mu|$-measurable).
When $H\subseteq G$, write $ \mu \conv H = \{ \mu \conv h\! : h \iin H \}$.
When $D\subseteq\MeasG$ and $H\subseteq G$, write
$ D \conv H = \{ \mu \conv h\! : \mu\iin D,\, h \iin H \}$.
Recall that $\delta(G) = \{ \delta_x\! : x\iin G \}$,
\;$\cl{\delta(G)}$ denotes the \wstar closure in $\MeasG^{\ast\ast}$,
giving a \wstar compact subset of its unit ball.

\begin{lemma}
    \label{lemma:fextension}
Let $G$ be any locally compact group.
Let $\{ D_\gamma\!: \gamma\in\Gamma \}$ be a family of subspaces
of $\MeasG$ such that $D_\beta \perp D_\gamma$ when
$\beta , \gamma \in \Gamma$, $\beta \neq \gamma$. 
If $h_\gamma\in \unitball (D_\gamma^\ast)$ is given for each
$\gamma\iin\Gamma$\,,
then there exists $h \in \unitball\bigl(\MeasG^\ast\bigr)$
that agrees with $h_\gamma$ on $D_\gamma$ for every $\gamma\iin\Gamma$.
\end{lemma}
\noindent
As a special case (taking $D_x=\lspan\{\delta_x\}$ for $x\iin G$), one can see
that the set $\delta(G)$ is weakly
discrete in $\MeasG$ and thus \wstar discrete in $\MeasG^{\ast\ast}$.

\Proof
If $\mu$ is in  the space $D'$ spanned by $\bigcup_\gamma D_\gamma$, then
 $\mu = \sum_{k=1}^n  \mu_k $
where $\mu_k \iin D_{\gamma_k} $
and $\gamma_j \neq \gamma_k$ for $j\neq k$.
There is then a unique linear functional $h'$ on $D'$
extending each $h_D$\,.
Note that if $\mu',\mu''\in\MeasG$ and $\mu'\perp\mu''$ then
$\norm{\mu'+\mu''} \;=\; \norm{\mu'} + \norm{\mu''} $ and hence
\[
\abs{\langle h',\mu \rangle }
\;\leq\; \sum_k \; \abs{\langle h',\mu_k\rangle}
\;\leq\; \sum_k  \norm{h_{\gamma_k}} \; \norm{\mu_k}
\;\leq\; \sum_k  \norm{\mu_k}
\;=\; \norm{\mu}\,.
\]
This shows that the norm of $h'$ is at most $1$ and thus $h'$ extends to
a linear functional $h \in \unitball\bigl(\MeasG^\ast\bigr)$.
\end{proof}

\begin{lemma}
    \label{lemma:thinAvoidance}
Let $G$ be any locally compact group, $\tau$ uncountable.
Assume that $M_0$ is a subspace of \,$\MeasG$ such that
if $\mu\iin M_0$ then $\abs{\mu}\iin M_0$ and $\mu$ is $\tau$-thin.
If $F\subseteq M_0$ is finite, $D \subseteq \MeasG$ and $\card{D}<\tau$\,,
then there exists \,$x\iin G$ such that \,$D\perp (F\conv x)$.
\end{lemma}

\Proof
The measure $\nu = \sum_{\xi\in F}\,\abs{\xi}$ \,is $\tau$-thin,
hence there is a set $P\subseteq G$ such that
$\card{P}=\tau$ and $\nu\conv p\,\perp\,\nu\conv p'$ for all
$p,p'\iin P$ with $p\neq p'$.
For every $\mu\iin D$ the set $P_\mu = \{p\iin P\! : \mu
\text{ is not singular to } \nu\conv p \,\}$
must be countable, because $\mu$ is a finite measure. Thus
\,$\card{\,\bigcup_{\mu\in D} P_\mu}\leq\card{D}\cdot\aleph_0<\tau=\card{P}$,
so that there exists \,$x \in P \setminus \bigcup_{\mu\in D} P_\mu$\,. Then
$D\perp\,\nu\conv x$ \,and therefore $D\perp (F\conv x)$.
\end{proof}

\begin{lemma}
    \label{lemma:shifts}
Let $G$ be any locally compact group, $\tau$ uncountable.
Assume that $M_0$ is a subspace of $\MeasG$ such that
if $\mu\iin M_0$ then $\abs{\mu}\iin M_0$ and $\mu$ is $\tau$-thin.
If $\{ F_\gamma\!: \gamma\in\Gamma \}$ is a family of finite subsets of $M_0$
and $\card{\Gamma}\leq \tau$,
then there exist $x_\gamma \iin G$ for $\gamma\iin\Gamma$ such that
\,$(F_\beta \conv {x_\beta} ) \perp (F_\gamma \conv {x_\gamma} ) $
when $\beta , \gamma \in \Gamma$, $\beta \neq \gamma$.
\end{lemma}

\Proof
Construct $x_\gamma$ by transfinite induction,
using Lemma~\ref{lemma:thinAvoidance} at each step.
\end{proof}

We say that a direct sum $M_2=M_0\oplus M_1$ of subspaces of $\MeasG$ is
{\it $G$-invariant}, if  \,$M_i\conv G\subseteq M_i$ holds for $i=0,1$
\;(which implies $M_2\conv G\subseteq M_2$).

\begin{lemma}
    \label{lemma:dense}
Let $G$ be a locally compact group,
$M_2 = M_0 \oplus M_1$ a $G$-invariant topological direct sum of closed
subspaces of \,$\MeasG$.
Let $\scrO$ be a collection of  \wstar open subsets of
$M_0^\ast$, each of them having non-empty intersection with
\,$\unitball (M_0^\ast)$.
\,$\tau=\card{\scrO}$ shall be uncountable.
Assume that $\mu\iin M_0$ implies that $\abs{\mu}\iin M_0$ and that
$\mu$ is \,$\tau$-thin.
\\
Then there exists \,$h\iin M_1^\circ$ such that the (projected) orbit
\,$p_0\bigl(\delta({G}) \odot h\bigr)$
intersects every set from~$\scrO$ and \,$h|M_0\iin\unitball (M_0^\ast)$.
\end{lemma}\noindent
As above,
$p_0\!:\MeasG^\ast\to \MeasG^\ast/M_0^\circ$ denotes the canonical
projection. Under the identification of $\MeasG^\ast/M_0^\circ$ with
$M_0^\ast$, we have $p_0(h)=h|M_0$ (restriction of the functional).

\Proof
Without loss of generality assume that each $U\iin\scrO$ is a basic
neighbourhood of the form
\[
U = \{ f\iin M_0^\ast :\;
\abs{\,\langle f - g_U ,\,\mu\,\rangle} < \varepsilon_U
\text{ for all \,} \mu\iin F_U \}
\]
where $F_U \subseteq M_0$ is a finite set,
$g_U \in \unitball (M_0^\ast)$ and $\varepsilon_U > 0$.
Apply Lemma~\ref{lemma:shifts} with $\Gamma = \scrO$ to obtain elements
$x_U \iin{G}$ for $U\iin\scrO$, such that
if \,$U,V\iin\scrO$, $U\neq V$, then
$(F_U \conv {x_U} ) \perp (F_V \conv {x_V} ) $\,.

Then apply Lemma~\ref{lemma:fextension}, taking \,$\Gamma = \scrO$\,,
\,$D_U$ the space spanned by \,$F_U \conv {x_U}$ and the
functionals \,$h_U \iin\unitball (D_U^\ast)$ defined by
\,$\langle h_U ,\nu\rangle = \langle\,g_U\,,\nu\conv{x_U^{-1}}\rangle$.
Thus there is\linebreak
$h\in \unitball (M_0^\ast)$ that agrees with $h_U$ on $D_U$ for
every $U\iin\scrO$ and we may extend it to $M_2$ so that $h=0$ on $M_1$\,.
Extending $h$ further to $\MeasG$\,,
this means that $h\iin M_1^\circ$\,, $h|M_0\in\unitball (M_0^\ast)$
\,and for $\mu\iin F_U$\,, we have
\[
\langle \delta_{x_U}\odot h,\,\mu\,\rangle
= \langle\,h\,,\mu\conv{x_U}\rangle
= \langle\,h_U ,\mu\conv{x_U}\rangle
= \langle\,g_U ,\mu\conv{x_U} \conv {x_U^{-1}}\rangle
= \langle g_U ,\mu\rangle \; ,
\]
hence \,$p_0(\delta_{x_U}\odot h)\in U$\,.
\end{proof}
\begin{Rem}
The argument gets somewhat more transparent if $M_0\perp M_1$ \,(which
will always be the case in the applications below). Then $M_0^\ast$
is {\it isometrically} isomorphic to $p_2(M_1^\circ)$ \,(these are the
functionals on $M_2$ vanishing on $M_1$) and the construction
above gives \,$h\iin\unitball (M_1^\circ)$\,.

The method of proof above shows a slightly stronger statement: \
Assume that given are finite dimensional subspaces $D_{\gamma}$ of $M_0$ and
functionals $h_\gamma\!\in \unitball (D_\gamma^\ast)$ for $\gamma\iin\Gamma$
such that
$\tau=\card{\Gamma}$ is uncountable, and $M_0,M_1$ are as in the lemma.
Then there exists \,$h\iin M_1^\circ$ with \,$h|M_0\iin\unitball (M_0^\ast)$
and such that for every $\gamma\iin\Gamma$ there exists
\,$h'_\gamma\in \delta({G}) \odot h$ satisfying
\,$h'_\gamma|D_\gamma=h_\gamma$\,.
\end{Rem}

\begin{theorem}[Factorization for thin measures]
    \label{th:factorization}
Let $G$ be a locally compact group,
$\MeasG = M_0 \oplus M_1$ a $G$-invariant topological direct sum such that
$d(M_0)$ is uncountable.
Assume that $\mu\iin M_0$ implies that $\abs{\mu}\iin M_0$ and that
$\mu$ is \,$d(M_0)$\,-thin.
Then there exists \,$h\iin M_1^\circ$ such that
\;$\cl{\delta(G)} \odot h \supseteq \unitball (M_1^\circ)$\,.
\end{theorem}\noindent
In fact, our argument produces an $h\in\unitball (M_0^\ast)$ with
\;$\cl{\delta(G)} \odot h = \unitball (M_0^\ast)$\,.

\Proof
$d(M_0)$ refers to the norm topology. Taking a norm-dense subset of $M_0$
with cardinality $d(M_0)$,
one can find a family $\scrO$ of \wstar open subsets of $M_0^\ast$
whose intersections with $\unitball (M_0^\ast)$ give a \wstar open base
and such that $\card{\scrO}=d(M_0)$. Now take $h\iin M_1^\circ$
as given by Lemma~\ref{lemma:dense}. The set \,$\cl{\delta(G)}\odot h$ is
\wstar closed
by Lemma~\ref{lemma:wcontinuity}. Note that $h|M_0\in \unitball (M_0^\ast)$
implies \,$p_0(\delta(G) \odot h) \subseteq \unitball (M_0^\ast)$
as in part~(1) of Lemma~\ref{lemma:wcontinuity}.
Then it follows easily from the
properties of $h$ and $\scrO$ that
\,$p_0\bigl(\cl{\delta(G)} \odot h \bigr)= \unitball (M_0^\ast)$\,.

For $h_1\iin\unitball (M_1^\circ)$, we have
\,$p_0(h_1)\in\unitball (M_0^\ast)$.
Hence there exists \,$h_2\iin\cl{\delta(G)}\odot h$ such that
$p_0(h_2)=p_0(h_1)$. $G$-invariance of $M_1$ implies $h_2\iin M_1^\circ$ and
it follows that $h_2=h_1$\,.
\end{proof}

\begin{corollary}
    \label{cor:thincentre}
Let $G$ be a locally compact group, $\MeasG = M_0 \oplus M_1$
a $G$-invariant topological direct sum of closed subspaces such that
$d(M_0)$ is uncountable.
Assume that $\mu\iin M_0$ implies that $\abs{\mu}\iin M_0$ and that
$\mu$ is \,$d(M_0)$\,-thin.
Then one gets\linebreak
$Z_t\bigl(\MeasG^{\ast\ast}\bigr)\subseteq M_0\oplus M_1^{\ast\ast}$.
\end{corollary}

\Proof
This follows immediately from Corollary~\ref{cor:centresum} in
combination with Theorem~\ref{th:factorization}.
\end{proof}

\begin{Rems}
In this section, we use several times the assumption that
$\mu\iin M_0$ implies $\abs{\mu}\iin M_0$\,. For a closed subspace
$M_0$ of $\MeasG$ this is in fact equivalent to $M_0$ being a
vector sublattice (i.e., in the complex case, $\mu\iin M_0$ implies that
its real and imaginary part belong to $M_0$ and the set of real measures in
$M_0$ is closed under the lattice operations).

In our paper, thinness of measures is a strong tool for non-separable spaces.
Formally, Theorem~\ref{th:factorization} and Corollary~\ref{cor:thincentre}
stay true without the condition that $d(M_0)$ should be uncountable (i.e.,
$M_0$ non-separable). But observe that there are no non-zero {\it separable}
$G$-invariant closed sublattices $M_0$ in $\MeasG$ such that every $\mu$ in
$M_0$ is \,$d(M_0)$\,-thin\,! (\,Let $\{\mu_n\!:\,n\ge0\}$ be a dense subset
of $\unitball (M_0)$ and put $\mu = \sum_{n=0}^\infty |\mu_n|/2^{n+1}$\,.
Then it is easy to see that $\nu\ll\mu$ holds for every \,$\nu\iin M_0$\,.
Thus, by $G$-invariance of $M_0$\,, it follows that
$\mu$ is not even $2$-thin, unless $\mu=0$\,). See also the comment following
Theorem \ref{th:perfectset} below.
\end{Rems}

Again there is an extended version, needed for the proof of our Main Theorem
in the non-metrizable case.

\begin{theorem}[Factorization on subspaces]
    \label{th:factorizationsub}
Let $G$ be a locally compact group,
$M_2 = M_0 \oplus M_1$ a topological direct sum of closed subspaces of
$\MeasG$ and let $\tau\ge d(M_0)$ be uncountable.
Assume that $\mu\iin M_0$ implies that $\abs{\mu}\iin M_0$ and that
$\mu$ is \,$\tau$-thin.
Furthermore, assume that there exists a $G$-invariant topological direct sum
\,$\widetilde M_2 = \widetilde M_0 \oplus \widetilde M_1$ of closed subspaces
of $\MeasG$ such that $M_i\subseteq\widetilde M_i$ holds for $i=0,1$. 
\\[.5mm]
Then there exists \,$h\iin M_1^\circ$ such that
\,$p_2\bigl(\cl{\delta(G)} \odot h\bigr) \supseteq
p_2\bigl(\unitball (M_1^\circ)\bigr)$ and
it follows that
$Z_t\bigl(\MeasG^{\ast\ast}\bigr)\cap M_2^{\ast\ast}\subseteq
M_0\oplus M_1^{\ast\ast}$.
\end{theorem}

\Proof
This is similar as above. There is a canonical projection
\,$q\!:(\widetilde M_0)^\ast\to M_0^\ast$\,, using restriction.
As above, take a family $\scrO$ of \wstar open subsets of $M_0^\ast$
whose intersections with $\unitball (M_0^\ast)$ give a
\wstar open base and such that $\card{\scrO}=\tau$.
Put \,$\widetilde\scrO=\{q^{-1}(U)\!: U\iin\scrO\,\}$.
Now take \,$h\in (\widetilde M_1\bigr)^\circ\subseteq M_1^\circ$
obtained from Lemma~\ref{lemma:dense}, applied to $\widetilde\scrO$ \,and
$\widetilde M_2 = \widetilde M_0 \oplus \widetilde M_1$\,.
It follows immediately that the projected orbit
\,$p_0\bigl(\delta({G}) \odot h\bigr)$
intersects every set from~$\scrO$ and as above this implies
\,$p_2\bigl(\cl{\delta(G)} \odot h\bigr) \supseteq
p_2\bigl(\unitball (M_1^\circ)\bigr)$.
The conclusion on $Z_t$ follows now from Lemma~\ref{lemma:centresubsp2}.
\end{proof}\vspace{1mm}

\section{Separation of singular measures}\vspace{1mm}
    \label{sec:singularSeparation}

Let $G$ be a locally compact group.
Then \,$\sMeasG\conv G\subseteq\sMeasG$
and \,$\aMeasG\conv G\subseteq\aMeasG$.
In view of the Lebesgue decomposition~\cite[19.20]{Hewitt1963aha},
we have \,$\MeasG=\sMeasG\oplus\aMeasG$.
Lemma~\ref{lemma:dense} and the other results of
Section~\ref{sec:factorization}
apply whenever every measure in $\sMeasG$ is $\card G$\,-thin (recall from
the introduction that
$\card G=d(\MeasG)$ holds for all infinite groups $G$);
in this section we show that this is the case when $G$ is metrizable
or more generally, when $\card{G}=2^{\aleph_0}\kappa(G)$. Further applications
will follow for the subspaces $\sMeas(G,K)$ of $\MeasG$.

Versions of the following lemma are well-known.
Saks~(\cite{Saks1937toi},~III.11) gives a proof for $G=\Rgroup^n$
and provides references to original sources.
A stronger version for $G=\Tgroup$
(where $\Tgroup=\Rgroup/{\mathbb Z}$ is the circle group) is proved by
Prokaj~\cite[Th.\,1]{Prokaj2003csm}.
If $G$ is a discrete group then $\sMeasG$ is the null space $\{0\}$.
Thus the lemma is of interest only for non-discrete groups,
although formally it holds for discrete groups as well.

\begin{lemma}
    \label{lemma:smalltranslation}
Let $G$ be a locally compact group.
If \,$\mu\iin\sMeasG$ and $U$ is any compact neighbourhood of $e_G$ then
\,$\mu\perp(\mu\conv x )$ for $\lambda_G$-almost all $x$ in $U$.
\end{lemma}

\Proof
Since the support of $\mu$ is always $\sigma$-compact, we may (replacing
$G$ by some open subgroup) assume $\sigma$-compactness of $G$\,.
Put $\lambda=\lambda_G$\,. Since $\mu\perp\lambda$ if and only if
$\abs{\mu}\perp\lambda$\,,
we may assume that $\mu\geq 0$ and $\mu\neq0$.

There is a $\mathsf{G_\delta}$\,-set $E\subseteq G$ such that
$\mu(G\setminus E)=0$ and $\lambda(E)=0$. Define
\[
f\!:G\times G\rightarrow[0,1]\qquad \text{ by }\qquad
f(x,y) = \left\{ \begin{array}{ll}
                 0 & \mathrm{if}\ \, yx\notin E                 \\
                 1 & \mathrm{if}\  \, yx\in E
              \end{array}
      \right.
\]
As $E$ is a $\mathsf{G_\delta}$\,-set, the function $f$ is Borel measurable on
$G\times G$\,.

Then \,$0\leq \int_U f(x,y) d\lambda(x) = \lambda(U\cap\,y^{-1} E) \leq
\lambda(y^{-1} E)=0$
\,for every $y\iin G$, because $\lambda(E)=0$.
On the other hand, we have
\,$\int_G f(x,y) d\mu(y) = \mu(E x^{-1} ) = \mu\conv x(E)$ \,for every
$x\iin G$.
Now apply Fubini's theorem for non-negative functions (being valid in the
$\sigma$-compact case  \cite[Thm.\,13.9]{Hewitt1963aha}) to get
\[
\int_U \mu\conv x\,(E) \, d\lambda(x)
= \int_U \int_G f(x,y) \, d\mu(y) \, d\lambda(x)
= \int_G \int_U f(x,y) \, d\lambda(x) \, d\mu(y)
= 0 \; .
\]
This implies that \,$\mu\conv x(E)=0$
for $\lambda$-almost all $x$ in~$U$.
Clearly, $\mu\perp(\mu\conv x )$ for every~$x$ for which $\mu\conv x(E)=0$.
\end{proof}

\begin{lemma}
    \label{lemma:compapprox}
Let $G$ be any locally compact group.
Let $\mu\iin\MeasG$, $\varepsilon>0$, and let $H\subseteq G$ be a countable
set such that the measures \,$\mu\conv h$ are pairwise mutually
singular for $h \iin H$. Then there is a compact set $C\subseteq G$ such that
\,$\abs{\mu}(G\setminus C) < \varepsilon$ and the sets
\,$C h$  are pairwise disjoint for $h \iin H$.
\end{lemma}

\Proof
Using singularity, we get
pairwise disjoint sets $E_h \subseteq G$ ($h\iin H$) such that
$\mu\conv h$ is concentrated on $E_h$  for all $h\iin H$
and \,$\abs{\mu\conv h'}(E_h ) = 0$ for all $h,h'\iin H$ with $h\neq h'$.

Put \,$E= \bigcap_{h\in H} E_h h^{-1} $.
Since \,$\abs{\mu}(G \setminus E_h h^{-1}) =
\abs{\mu\conv h}(G\setminus E_h ) = 0$,
it follows that $\abs{\mu}(G\setminus E)=0$.
Thus there is a compact set $C\subseteq E$
such that \,$\abs{\mu}(G\setminus C) < \varepsilon$.
The sets $C h$ are pairwise disjoint for $h \iin H$ because
\,$Ch \subseteq E_h$\,.
\end{proof}

\begin{Rem}
All that is actually needed in Lemma \ref{lemma:compapprox} is that
$H$ has cardinality less than the additivity of the measure $\mu$
--- the least cardinal of a family of null sets whose union is not a null set.
\end{Rem}

In the next proof (and further on in Theorem~\ref{th:separationss}), the
formalism of ordinals will be used in the way developed by von Neumann ---
in particular, every ordinal is equal to the set of its predecessors.
For example, this means that for ordinals $\alpha$ and~$\beta$ the assertions
$\alpha < \beta$\,, $\alpha \subsetneq \beta$ and $\alpha \in \beta$ have the
same meaning. In particular, for $n\in \Nnat=\{ 0,1, 2, \ldots\}$, the
equality $n = \{ 0,1, 2, \ldots, n-1\}$ holds.
\\[1mm]
The following result was proved by Prokaj~\cite[Theorem 10]{Prokaj2003csm}
for $G=\Rgroup$\,.

\begin{theorem}
    \label{th:perfectset}
Let $G$ be any non-discrete locally compact group, $\mu\in\sMeasG$.
Then there exists a $\mathsf{K_\sigma}$-set
$E\subseteq G$ and a set $P\subseteq G$ such that $\mu$ is concentrated
on~$E$\,,
$\card{P}=2^{\aleph_0}\kappa(G)$ and \,$(Ep)\cap(Ep')=\emptyset$ for all
$p,p'\iin P$ with $p\neq p'$.
Thus, every $\mu\iin\sMeasG$ is \,$2^{\aleph_0}\kappa(G)$\,-thin.
\end{theorem}

\noindent
In particular, we recover the result of \cite{Larsen}: \ if \,$\mu\iin\MeasG$
and the orbit $\{\mu\conv x\!: x\iin G\}$ is (norm-)\,separable, then
$\mu\iin\aMeasG$. See also \cite{Glick} for related results.
In fact, this shows once again that in the previous section our permanent
assumption ``$d(M_0)$ uncountable" makes no restriction (even when $M_0$ is
not a sublattice), as long as singular measures are considered.
\\
If $G$ is metrizable, then (as in \cite{Prokaj2003csm}) the set $P$
can be chosen to be perfect (in the proof below, just add the requirement
that the diameter of the sets $U_j$ be less than $1/j$ for $j>0$).

\Proof
For $x_j\iin G$\,  and a finite subset $d=\{j_0,\dots,j_l\}$
of $\Nnat$\,, where $j_0<j_1<\dots$\,, we define
$d_*=x_{j_0}\cdots x_{j_l}$  (with $d_*=e_G$ when $d$ is empty).\vspace{.5mm}

Let $\mu\iin\sMeasG$ and assume, without loss of generality, that
$\mu\geq0$ and $\mu\neq0$.
For $j\iin\Nnat$\,, we will construct by induction
\begin{itemize}
\item
compact neighbourhoods $U_j$ of $e_G$
\item
compact sets $C_j \subseteq G$
\item
elements $x_j \iin G$,
\end{itemize}
so that the following conditions $(1^\circ) - (5^\circ)$
hold\,:
\item[$(1^\circ)$]
$(\mu \conv  d_*) \perp ( \mu \conv d'_* ) $ \
whenever $d$ and $d'$ are distinct subsets of  $j$\,,
\item[$(2^\circ)$]
$\mu(G\setminus C_j) \leq 2^{-j}$\,,
\item[$(3^\circ)$]
$( C_j \: d_*\, U_{j+1} ) \;\cap\; ( C_j \: d'_*\, U_{j+1} ) = \emptyset $
\ whenever $d$ and $d'$ are distinct subsets of  $j$\,,
\item[$(4^\circ)$]
$ U_{j+1} \cap (x_j U_{j+1} ) = \emptyset $\,,
\item[$(5^\circ)$]
$ U_{j+1} \cup ( x_j U_{j+1} ) \subseteq U_{j} $\,.\vspace{1mm}

Let $U_0$ be any compact neighbourhood of $e_G$\,.
Let $n \iin \Nnat$ be such that $U_n$ has been constructed as well as $C_j$
and $x_j$ for all $j< n$ so that $(1^\circ) - (5^\circ)$ hold.

If $n=0$, put $\nu=\mu$\,, otherwise
$ \nu = \sum_{d\subseteq n}\mu \conv d_* $\,.
Note that $\nu\iin\sMeasG$
and, by Lemma~\ref{lemma:smalltranslation}, there is $x_n$ in  the interior of
$U_n$ such that $\nu\perp(\nu\conv x_n)$\,.
From that and from $(1^\circ)$ for $j<n$ it follows that  $(1^\circ)$ holds
for $j=n + 1$ as well.

By Lemma~\ref{lemma:compapprox} there is a compact set
$C_n \subseteq G$ such that $(2^\circ)$
holds for $j=n$\, and
\[
( C_n \, d_*)\, \cap\, ( C_n \, d'_* )\; =\; \emptyset
\mathrm{\quad \ for\ distinct \;\;} d,d'\subseteq n+1\,.
\]
Thus $ \{ C_n \, d_* \!: d \subseteq{n+1} \} $
is a finite family of pairwise disjoint compact sets,
and there is a neighbourhood $W$ of $e_G$ such that
\[
( C_n \, d_*\, W ) \;\cap\; ( C_n \, d'_*\, W )\; =\; \emptyset
\mathrm{\quad \ for\ distinct \;\;} d,d'\subseteq n+1\,.
\]
Since $x_n$ is in the interior of $U_n$ and $x_n \neq e_G$\,,
there is a compact neighbourhood $U_{n+1}$ of $e_G$ such that
$U_{n+1} \subseteq W$ and $(4^\circ)$ and $(5^\circ)$ hold with $j=n$.
Since $U_{n+1} \subseteq W$, we get also $(3^\circ)$ for $j=n$.
The construction of $U_j$, $C_j$ and $x_j$
satisfying $(1^\circ) - (5^\circ)$ is complete.\vspace{1mm}

Next, for each $n\iin \Nnat$ define \,$E_n = \bigcap_{j = n}^\infty C_j$
and put $E = \bigcup_n E_n$\,.
Then $E$ is a $\mathsf{K_\sigma}$-set and $\mu(G\setminus E)=0$.
For  $d\subseteq\Nnat$ \,(not necessarily finite), we define
\;$K(d)=\bigcap_{n=1}^\infty \;  (d\cap n)_\ast\,U_n \ . $

Note (since $(5^\circ)$ implies
$(d\cap n)_\ast\,U_n\subseteq(d\cap m)_\ast\,U_m$ for $m<n$\,) that
$K(d)\subseteq U_0$ is nonempty, being the
monotone intersection of compact sets.
From $(4^\circ)$ it follows $K(d)\cap K(d')=\emptyset$ for $d\neq d'$.
Form $P_0$ by taking one element in each $K(d)$.
Thus $\card{P_0}=2^{\aleph_0}$. Let $H$ be the subgroup generated
by $E\,U_0$\,. We take $P_1$ a set of representatives for the right cosets in
$G$ (with respect to $H$) and put $P=P_0P_1$. Since $H$ is open and
$\sigma$-compact,
we have $\card{P_1}=\kappa(G)$ when $G$ is not $\sigma$-compact. Then
$\card{P}=2^{\aleph_0}\kappa(G)$ follows for every $G$.
It remains to be proved that $(Ep)\cap(Ep')=\emptyset$ for
$p,p'\iin P_0$ with $p\neq p'$.

Take $e,e'\iin E$ and $p,p'\iin P_0$, $p\neq p'$.
By the definition of $E$ and $P_0$, we have $e,e' \iin E_n \subseteq C_n$ for
some~$n$,
and $p\in\!K(d)$, $p'\iin K(d')$ for some distinct subsets $d,d'$ of~$\Nnat$.
Let $j\geq n$ be such that \,$d\cap j\,\neq\, d' \cap j$.
Since \,$p\in (d\cap j)_\ast U_{j + 1}$\,,
$p'\iin (d'\cap j)_\ast U_{j + 1}$
and $e,e'\iin E_n \subseteq C_j$\,,
we get from $(3^\circ),\ ep \neq e'p'$.
\end{proof}\vspace{-1mm}
\noindent
(For $\kappa(G) > 2^{\aleph_0}$, the inductive construction is in fact
redundant [if one does not need that $E$ be perfect],
just take $E$ to be the support of $\mu\,,\ P_0=\{e_G\}$ and
$P_1,P$ as above\,).
\vspace{1mm}

\section{Proof of the Main Theorem --- metrizable case}\vspace{1mm}
    \label{sec:caseI}

\begin{lemma}
    \label{lemma:Lonecase}
Let $G$ be a locally compact group.
Then we have
$$Z_t\bigl(\MeasG^{\ast\ast}\bigr)\cap\aMeasG^{\ast\ast}\subseteq\; \MeasG.$$
\end{lemma}

\Proof
Since $\aMeasG$ is a subalgebra of $\MeasG$, it follows by elementary
arguments that
\,$Z_t\bigl(\MeasG^{\ast\ast}\bigr)\cap\aMeasG^{\ast\ast}\subseteq
Z_t\bigl(\aMeasG^{\ast\ast}\bigr)$. Now one can apply the theorem, due to Lau
and Losert \cite[Theorem 1]{LauLosert}
that \,$Z_t\bigl(\aMeasG^{\ast\ast}\bigr)=\aMeasG$.
\end{proof}

\begin{theorem}\label{th:caseI}
Let $G$ be a locally compact group, $K$ a compact subgroup such that
$\card{G/K}\le 2^{\aleph_0}\kappa(G)$.
Then \,$Z_t\bigl(\MeasG^{\ast\ast}\bigr)\cap\Meas(G/K)^{\ast\ast}\subseteq
\Meas(G/K)$.
\end{theorem}\noindent
As mentioned in Section~\ref{sec:notation}, the assumption on $G/K$ is
satisfied if $G/K$ is\linebreak metrizable.
The case of the trivial subgroup $K=\{e_G\}$ gives the Main Theorem,
i.e., $Z_t\bigl(\MeasG^{\ast\ast}\bigr)=\MeasG$ holds when
$\card G\le 2^{\aleph_0}\kappa(G)$ \;(in particular for metrizable groups).
\Proof
Put $M_0=\sMeas(G,K)$, $M_1=\aMeas(G,K)$, $M_2=\Meas(G/K)$. If $K=\{e_G\}$,
\linebreak we can (using Theorem~\ref{th:perfectset})
apply Corollary~\ref{cor:thincentre} and get that
\ $Z_t\bigl(\MeasG^{\ast\ast}\bigr)\subseteq\linebreak
\sMeas(G)\oplus\aMeasG^{\ast\ast}$. 
Since $\MeasG\subseteq Z_t\bigl(\MeasG^{\ast\ast}\bigr)$, the conclusion
now follows from Lemma~\ref{lemma:Lonecase}.

In the general case, we take \,$\widetilde M_0=\sMeasG$,
$\widetilde M_1=\aMeasG$, $\widetilde M_2=\MeasG$ and apply now
Theorem \ref{th:factorizationsub}.
\end{proof}\vspace{1mm}

If $G$ is any locally compact group and $G_d$ denotes the same group with
discrete topology, then Theorem~\ref{th:caseI} applies to $G_1=G\times G_d$
which contains $G$ as an open subgroup. Unfortunately, we do not have a
direct argument that strong Arens irregularity of the measure algebra
carries over to open subgroups.
If one assumes that
$2^{\aleph_1} = 2^{\aleph_0}$ (which is consistent with standard set theory),
then Theorem~\ref{th:caseI} applies to groups $G=\prod\limits_{i\in I}G_i$
with $G_i$ compact metrizable and $\card I\le\aleph_1$\,, giving examples
of compact non-metrizable groups.
However, the proof of the full conjecture requires considerably more work.
Theorem~\ref{th:caseI} (for $G/K$ metrizable)
will provide the starting point of an inductive argument.

\section{Compact subgroups and some classes of measures}\vspace{1mm}
    \label{sec:compsub}

First, we will give now some formulas for the character of quotient spaces.
Then we define some classes of compact subgroups in a non-metrizable group,
corresponding classes of measures and decompositions of $\MeasG$. This
contains the strongly singular measures mentioned at the beginning.

\begin{lemma}
    \label{lemma:basicproperties}
Let $G$ be a locally compact group.\vspace{1mm}
\item[(1)]
If $H_0$ and $H_1$ are closed subgroups of $G$\,, $H_0 \supseteq H_1$\,, then
\[
\chi(G/H_1)\; =\; \max\bigl(\chi(G/H_0 ), \chi(H_0 /H_1 )\bigr) \; .
\]
\item[(2)]
If $H_i$ are closed subgroups of $G$ for $i\in I$, then
\[
\chi\bigl(\,G\big/{\textstyle\bigcap\limits_{i\in I}}H_i\,\bigr)\;\leq\;
\sup_{i\in I}\bigl(\,\chi(G/H_i)\,\bigr)+\card{I} \; .
\]
\item[(3)]
If $G$ is $\sigma$-compact, $H$ a closed subgroup,
\,$N=\bigcap_{x\in G}\,xHx^{-1}$, then
\[
\chi(G/N)\; \leq\; \chi(G/H)+\aleph_0\,.
\]
\end{lemma}\noindent
In \thetag{3}, without $\sigma$-compactness one has \;
$\chi(G/N) \leq \chi(G/H)+\kappa(G)$\,.

\Proof
Recall that in a Hausdorff space $\Omega$ a family $(U_j)_{j\in J}$ of
compact neighbourhoods of $\omega\iin\Omega$ is a neighbourhood base iff it is
downwards directed and satisfies $\bigcap_{j\in J}\mspace{1mu}U_j=\{\omega\}$.
Thus, if $\Omega$ is locally compact, $\chi(\omega,\Omega)$ is the minimal
cardinality of a family of $\omega$-neighbourhoods whose intersection
is $\{\omega\}$.

To prove (1), take a family $\mathcal U_0$ of cardinality $\chi(G/H_0)$
consisting of open sets in $G$ with
$\bigcap\mspace{1mu}\mathcal U_0 = H_0$ and a family $\mathcal U_1$ of
cardinality $\chi(H_0/H_1)$ consisting of open sets in $G$ with
$\bigcap\mathcal U_1\cap  H_0 = H_1$\,. In addition, we assume
that \,$U=U\mspace{1mu}H_i$ for $U\!\iin\mathcal U_i$\,. Then the image of
$\mathcal U_0\cup\mathcal U_1$ under the quotient map to $G/H_1$ has $\{H_1\}$
as the intersection. Conversely, starting with a neighbourhood base in
$G/H_1$\,, let $\mathcal U$ be a family of open sets in
$G$ that is downwards directed and whose intersection is $H_1$\,.
We may assume that there is a compact subset
$C$ of $G$ such that $U=U\mspace{1mu}H_1\subseteq C\mspace{1mu}H_1$ for all
$U\iin\mathcal U$\,. Then one can show that
$\{UH_0\!:U\iin\mathcal U\,\}$ \;(resp. $\{U\cap H_0\!:U\iin\mathcal U\,\}$)
has intersection $H_0$ (resp. $H_1$).
Similarly for \thetag{2}.

To show \thetag{3}, we can (replacing $G$ by $G/N$) assume
that $N$ is trivial. Let $(V_i)_{i\in I}$ be a family of compact
$e_G$-neighbourhoods such that \,$\bigcap_{i\in I}\mspace{1mu}V_i\subseteq H$
and $\card{I}=\chi(G/H)$. Choose compact symmetric $e_G$-neighbourhoods
$W_i$ such that \,$W_i^3\subseteq V_i$ and countable subsets $D_i$ in $G$ such
that \,$G=D_iW_i$\,. Then the family of sets
$\{\,xW_ix^{-1}:\linebreak x\iin D_i\,,\,i\iin I\}$ has intersection $\{e_G\}$
and its cardinality is at most $\chi(G/H)+\aleph_0$\,. 
Alternatively, when $H$ is
compact, one can prove \thetag{3} by using from \cite{Hu}, Lemma 2 and
formulas\;\thetag{\dag} and \thetag{\ddag}\,(which extends to quotient
spaces).
\end{proof}

\begin{corollary}
    \label{corollary:basicproperties:monotone}
Let $\gamma>0$ be an ordinal number.
If $K_\alpha$ are compact subgroups of~$G$ for $\alpha\leq\gamma$\,, such that
\,$K_\alpha\supseteq K_{\alpha+1}$ and
\,$\chi(K_{\alpha}/K_{\alpha+1})=\aleph_0$
for all $\alpha<\gamma$\,, and
\,$K_\beta = \bigcap_{\alpha<\beta} K_\alpha$ \,for all limit ordinals
$\beta\leq\gamma$\,,
then
\;$\chi(G/K_\gamma) = \card{\gamma}+\chi(G/K_0)+\aleph_0$\,.
\end{corollary}

\Proof
By induction, Lemma\;\ref{lemma:basicproperties} gives
\,$\chi(G/K_\gamma) \le \card{\gamma}+\chi(G/K_0)+\aleph_0$\,.
In addition, equality follows from Lemma \ref{lemma:basicproperties}\,%
\thetag{1} when
\,$\card{\gamma}\le\chi(G/K_0)+\aleph_0$ \,(since $\gamma>0$ and
$\chi(K_{\alpha}/K_{\alpha+1})=\aleph_0$).
Now assume that equality does not hold for some $\gamma$ and choose
$\gamma$ minimal. Put $\tau=\chi(G/K_\gamma)$\,. Then $\tau<\card{\gamma}$
and \,$\card{\alpha}=\chi(G/K_\alpha)\le\tau$ \,if
\,$\chi(G/K_0)+\aleph_0\le\card{\alpha}$ and $\alpha<\gamma$\,. Thus
$\gamma$ must be a limit ordinal.
Since \,$K_\gamma = \bigcap_{\alpha<\gamma} K_\alpha$\,, it follows that the
family of sets
$VK_\alpha$ where $V$ is an $e_G$-neighbourhood and $\alpha<\gamma$
has the intersection $K_\gamma$\,. Thus (see the beginning of the proof of
Lemma \ref{lemma:basicproperties}) it defines a neighbourhood basis in
$G/K_\gamma$ \,(i.e., $G/K_\gamma$ is homeomorphic to the projective limit
of the spaces $G/K_\alpha$ where $\alpha<\gamma$\,).
This implies that there exists a family $(V_i)_{i\in I}$ of
$e_G$-neighbourhoods and $\alpha_i<\gamma$ such that
\,$\bigcap_{i\in I}V_iK_{\alpha_i}=K_\gamma$\,,
where $\card{I}=\tau$\,. Put
$\beta=\sup\{\,\alpha_i\!: i\iin I\}$. Since $\card{\alpha_i}\le\tau$ for
all $i$ and $\card{I}=\tau$\,, it follows that $\card{\beta}\le\tau$,
giving $\beta<\gamma$\,. But then we would have
\,$V_iK_{\alpha_i}\supset K_{\beta}$ for all $i$ which is impossible.
\end{proof}
\vspace{2mm plus 1mm}

Now assume that $G$ is a non-discrete locally compact group
(i.e., $\chi(G)\ge\aleph_0$\,; but our main interest will be the case that
$G$ is a non-metrizable). Let $\tau$ be a cardinal number with
\,$\aleph_0\le\tau\le\chi(G)$. We put
\begin{align*}
\scrK_\tau\,&=\,\{K\!: \,K\text{ compact subgroup of }G,\ \chi(G/K)\le\tau\}\\
\scrK_\tau^\circ\,&=\,\{K\!: \,K\text{ compact subgroup of }G,\ \chi(G/K)<\tau\}\;.
\end{align*}

\begin{corollary}
    \label{corollary:basicproperties:classes}
Let $G$ be a locally compact group, $\aleph_0\le\tau\le\chi(G)$. Assume
that $K_i\iin\scrK_\tau^\circ\ \,(i=1,\dots,n)$
and that $H$ is an open $\sigma$-compact subgroup of $G$. Then it follows
\,$\bigcap_{i=1}^n K_i\iin\scrK_\tau^\circ$\,,
$H\cap K_i\iin\scrK_\tau^\circ$\,, and there exists
$N\iin\scrK_\tau^\circ$ such that $N$ is a normal subgroup of $H$ and
\,$N\subseteq\bigcap_{i=1}^n K_i$\,.\vspace{-1mm}
\end{corollary}\noindent
Similar properties hold for $\scrK_\tau$\,. In fact,
$\bigcap_{i\in I} K_i\iin\scrK_\tau$ holds when $K_i\iin\scrK_\tau$ for
$i\iin I$ and $\card{I}\le\tau$\,. \;Observe that
\,$\scrK_\tau=\scrK_{\tau^+}^\circ$\,, where $\tau^+$ denotes the successor
cardinal of~$\tau$\,. Furthermore, if $H$ is any open subgroup
of a topological group $G$\,, then
\[
\chi(G/H_1)=\chi\bigl(H/(H_1\cap H)\bigr)=\chi\bigl(G/(H_1\cap H)\bigr)
\]
for every subgroup $H_1$ of $G$\,. Every $\mu\iin\MeasG$ is
supported by some open $\sigma$-compact\linebreak
subgroup (depending on $\mu$).
This will allow us to reduce many arguments to\linebreak
$\sigma$-compact groups \,(see Lemma\;\ref{lemma:heredit}).
\vspace{2mm plus 1mm}

We will now use these classes of compact subgroups to define some classes of
measures that will be useful in the induction.
\begin{align*}
\Meas_\tau(G)\;&=\;\bigcup_{K\iin\,\scrK_\tau}\Meas(G/K)\qquad
\text{measures of character $\tau$ .}
\\[-2mm]
\intertext{For $\tau=\aleph_0$ we put\quad
$\ssMeast{\aleph_0}(G)= \sMeasG\cap\Meas_{\aleph_0}(G)
\text{\ \ and \ }\aiMeast{\aleph_0}(G)=\aMeasG$,\vspace{1mm}\newline
and for $\aleph_0<\tau\le\chi(G)$\vspace{-3mm}}
\ssMeast{\tau}(G)\:&=\:\{\,\mu\in\sMeasG\cap\Meas_\tau(G):\;
\mu\perp\Meas_{\tau_1}(G)\text{ \,for all }\tau_1<\tau\}
\\
&\text{\kern 1cm strongly singular measures of character $\tau$}
\\
\aiMeast{\tau}(G)\,&=\,\{\mu\in\Meas_\tau(G):\;
\mu=\lim_{K\!\in\,\scrK_\tau^\circ} \mu\conv\lambda_K
\text{ \,(norm limit of the net)}\}
\vspace*{-2mm}
\end{align*}
\hspace{4cm}approximately invariant measures of character $\tau$
\\[.5mm]\hspace*{1.6cm}(by Corollary \ref{corollary:basicproperties:classes},
$\scrK_\tau^\circ$ is downwards directed under inclusion).
\\[2mm]
More generally, when $K$ is a compact subgroup of $G$ that is not open
(equivalently,
\,$\chi(G/K)\ge\aleph_0$\,),
we put \ $\ssMeas(G,K)=
\Meas(G/K)\cap\ssMeast{\chi(G/K)}(G)$ \ and
\ $\aiMeas(G,K)=\linebreak\Meas(G/K)\cap\aiMeast{\chi(G/K)}(G)$.
\ The definitions of $\ssMeas(G,K)$ and $\aiMeas(G,K)$ impose conditions
coming from the bigger space $\MeasG$, using the embedding of $\Meas(G/K)$
into $\MeasG$ described in Lemma \ref{Identify} \,(see also
Lemma \ref{lemma:approx} for a more intrinsic description in terms of $G/K$).
If $K$ is normal in $G$\,, things are easier, see
Corollary \ref{corollary:qIdentification} below.
\\[.5mm]
Note that if $\tau_1<\tau$\,, then \,$\scrK_{\tau_1}\subseteq\scrK_\tau$ and
\,$\Meas_{\tau_1}(G)\subseteq\Meas_\tau(G)$.

\begin{lemma}  \label{lemma:heredit}\vspace{1mm}
Let $G$ be a locally compact group, $K$ a non-open compact subgroup.
\item[(i)] If $H$ is an open subgroup of $G$ satisfying \,$K\subseteq H$\,,
then \quad$\aiMeas(H,K)\,=\hspace*{\fill}\linebreak
\aiMeas(G,K)\cap\Meas(H)$ and
\;$\ssMeas(H,K)=\ssMeas(G,K)\cap\Meas(H)$.
\item[(ii)] If $K'$ is a compact subgroup of $K$ with $\chi(G/K')=\chi(G/K)$,
then \;$\aiMeas(G,K)=\aiMeas(G,K')\cap\Meas(G/K)$ and
\;$\ssMeas(G,K)=\ssMeas(G,K')\cap\Meas(G/K)$.
\end{lemma}

\Proof
For a compact subgroup $K_1$ of $G$\,,
Corollary\;\ref{corollary:basicproperties:classes} gives
$K_1\cap H\iin\scrK_\tau$
iff\linebreak $K_1\iin\scrK_\tau$\,, and by Lemma\;\ref{Identify},
\,$\Meas(G/K_1)\cap\Meas(H)\subseteq\Meas\bigl(H/(K_1\cap H)\bigr)$.
It follows that \,$\Meas_\tau(G)\cap\Meas(H)=\Meas_\tau(H)$, leading to
corresponding
formulas for $\ssMeast{\tau}$ and $\aiMeast{\tau}$\,.  This implies \thetag{i}
\,(alternatively, one might use Lemma\;\ref{lemma:approx}).
\thetag{ii} follows immediately from the definitions.
\end{proof}
\vspace{1mm plus 2pt}

\begin{theorem}\label{th:measureclasses}
Let $G$ be a non-discrete locally compact group, $\tau$ a cardinal number
with \,$\aleph_0\le\tau\le\chi(G)$, \;$K$ a non-open compact subgroup of $G$. 
\vspace{1pt plus 2pt}
\item[(i)] $\Meas_\tau(G)$ and $\aiMeast{\tau}(G)$ are norm closed
ideals in $\MeasG$, and \,$\ssMeast{\tau}(G)$ is a norm closed
subspace.
\item[(ii)] $\Meas_\tau(G),\:\aiMeast{\tau}(G)$ and
$\ssMeast{\tau}(G)$ are L-subspaces \,(i.e., $\mu\iin M$ and
$\abs{\nu}\ll\abs{\mu}$ implies \,$\nu\iin M$). We have \
$\Meas_\tau(G)=\ssMeast{\tau}(G)\oplus\aiMeast{\tau}(G)$
\vspace{1pt plus 2pt}and
\;$\ssMeast{\tau}(G)\perp\aiMeast{\tau}(G)$.
\item[(iii)]  $\Meas(G/K),\;\aiMeas(G,K)$ and $\ssMeas(G,K)$ are closed
subspaces and vector sublattices in $\MeasG$ \,(i.e., $\mu\iin M$ implies
$\abs{\mu}\iin M$). We have \vspace{-3pt}
\[
\Meas(G/K)\;=\;\ssMeas(G,K)\oplus\aiMeas(G,K)
\text{ \quad and \quad} \ssMeas(G,K)\perp\aiMeas(G,K)\,.
\]
\end{theorem}

\Proof
$\Meas(G/K)$ is always a closed subspace of $\MeasG$ and a left ideal
(see Lemma~\ref{Identify}).
Using \thetag{2} of
Lemma \ref{lemma:basicproperties}, it follows that
$\Meas_\tau(G)$ is a closed subspace and then the same for
$\aiMeast{\tau}(G)$ and $\ssMeast{\tau}(G)$. This implies
(looking at the definitions) that
$\Meas_\tau(G)$ and $\aiMeast{\tau}(G)$ are left ideals.
Now take $\mu,\nu\iin\MeasG$ and let $H$ be an open $\sigma$-compact
subgroup containing the support of $\nu$. If $\mu\iin\Meas_\tau(G)$, then
by \thetag{2},\thetag{3} of
Lemma~\ref{lemma:basicproperties}, there exists a
normal subgroup $N$ of $H$ such that $N\iin\scrK_\tau$ and $\mu\iin\Meas(G/N)$.
As mentioned right after Lemma~\ref{Identify}, normality of $N$ implies
that $\Meas(G/N)$ is right \linebreak $H$-invariant. It follows that 
$\mu\conv\nu\iin\Meas(G/N)\subseteq\Meas_\tau(G)$. Thus $\Meas_\tau(G)$ is
a right ideal and a similar argument works for $\aiMeast{\tau}(G)$
\,(if $H$ is an open $\sigma$-compact subgroup of~$G$\,, the normal subgroups
in $H$ define a subset of $\scrK_\tau^\circ$ that is cofinal by
Lemma~\ref{lemma:basicproperties}\,\thetag{3}\,).
This proves \thetag{i}.

To show that $\Meas_\tau(G)$ is an L-subspace, it will be enough
(by closedness) to prove that $h\mu\iin\Meas_\tau(G)$ whenever
$\mu\iin\Meas_\tau(G)$ and $h$ is a continuous function of compact support.
As above let $H$ be an open $\sigma$-compact subgroup such that $h$
vanishes outside $H$. By the Kakutani-Kodaira theorem (see
Lemma\;\ref{lemma:KKthm} below), there exists a compact normal
subgroup $N$ of $H$ such that $H/N$ is metrizable and $h$ is $N$-periodic.
If $\mu\iin\Meas(G/K)$ with $K\iin\scrK_\tau$\,, we have
$h\mu\iin\Meas\bigl(G/(K\cap N)\bigr)$ and by 
\thetag{2} of Lemma~\ref{lemma:basicproperties}, $K\cap N\iin\scrK_\tau$\,.
It is easy to see that this also implies that $\ssMeast{\tau}(G)$
is an L-subspace. Similar arguments work for
$\aiMeast{\tau}(G)$ and show also that
\linebreak$\aMeasG\subseteq\Meas_{\aleph_0}(G)$.
\;$\ssMeast{\tau}(G)\perp\aiMeast{\tau}(G)$ follows from
the definition. Observe that
$\Meas_{\tau_1}(G)\subseteq\aiMeast{\tau}(G)$ holds for
$\tau_1<\tau$\,. Hence if $\mu\iin\Meas_\tau(G)$ and
$\mu\perp\aiMeast{\tau}(G)$ then $\mu\iin\sMeasG$ and combined
$\mu\iin\ssMeast{\tau}(G)$. Now take any $\mu\iin\Meas_\tau(G)$
with $\mu\ge0$. If there exists $\nu\iin\aiMeast{\tau}(G)$ such
that $\nu\ne0$ and $\nu$ is not singular to $\mu$, we can (decomposing and
taking advantage of the properties of an
L-subspace) assume that $0\le\nu\le\mu$\,. Then (using closedness and again
lattice properties) we can
find such a $\nu$ with $\norm{\nu}$ maximal and it follows that
$\mu-\nu\perp\aiMeast{\tau}(G)$. As mentioned above, this
implies $\mu-\nu\iin\ssMeast{\tau}(G)$. It follows that
$\Meas_\tau(G)=\ssMeast{\tau}(G)\oplus\aiMeast{\tau}(G)$
\ (for $\tau=\aleph_0$ one gets just the Lebesgue decomposition of elements
of $\Meas_{\aleph_0}(G)$ with respect\linebreak to $\lambda_G$).

The statements in \thetag{iii} about $\Meas(G/K)$ and its subspaces are easy
consequences of \thetag{i} and \thetag{ii}.
\end{proof}

\begin{corollary} \label{corollary:qIdentification}
If $K$ is a compact normal subgroup of $G$ that is not open, then the
isomorphism of Lemma\;\ref{Identify} maps the subspace \,$\ssMeas(G,K)$ onto
\,$\ssMeast{\chi(G/K)}(G/K)$ and \,$\aiMeas(G,K)$ onto
\,$\aiMeast{\chi(G/K)}(G/K)$.
\end{corollary}
\begin{proof}
Let $\Phi(\mu)=\dot\mu$ be this isomorphism ($\mu\iin\Meas(G,K)$\,).
As mentioned earlier (the image measure of Haar measure is Haar measure),
\,$\Phi\bigl(\sMeasG\cap\Meas(G,K)\bigr)=\sMeas(G/K)$. For a compact subgroup
$L\supseteq K$\,, we have \,$\Phi\bigl(\Meas(G,L)\bigr)=\Meas(G/K,L/K)$ and it
follows that \,$\Phi^{-1}\bigl(\Meas_\tau(G/K)\bigr)\subseteq\Meas_\tau(G)$
\,holds for \,$\aleph_0\le\tau\le\chi(G/K)$. \,$\Phi$~being order preserving,
this implies that
\,$\Phi^{-1}\bigl(\ssMeast{\chi(G/K)}(G/K)\bigr)\supseteq\ssMeas(G,K)$ \,and
similarly
\,$\Phi^{-1}\bigl(\aiMeast{\chi(G/K)}(G/K)\bigr)\supseteq\aiMeas(G,K)$. But
then by \thetag{ii},\;\thetag{iii} of Theorem~\ref{th:measureclasses}, these
inclusions must be equalities.
\end{proof}

\begin{Rems}
From \thetag{ii} of the last theorem, the following decomposition
results: \,Every $\mu\iin\sMeasG$ has a unique (up to
reorderings) representation $\mu=\sum\mu_i$\,, where \linebreak
$\mu_i\iin\ssMeast{\tau_i}(G)$ for some pairwise different cardinals
$\tau_i$ with \,$\aleph_0\le\tau_i\le\chi(G)$. Thus $\sMeasG$ is the
$l^1$-sum of all the spaces $\ssMeast{\tau}(G)$ with
\,$\aleph_0\le\tau\le\chi(G)$. Similarly, for $\aleph_0\le\tau\le\chi(G)$,
\,$\aiMeast{\tau}(G)$ is the $l^1$-sum of $\aMeasG$ and all the spaces
$\ssMeast{\tau'}(G)$ with $\aleph_0\le\tau'<\tau$.

As mentioned in the proof of \thetag{ii} above,
$\aMeasG\subseteq\Meas_{\aleph_0}(G)$ holds. It follows that in the
definition of $\ssMeast{\tau}(G)$, the condition $\mu\iin\sMeasG$
is redundant for $\tau>\aleph_0$\,.

For a totally disconnected group $G$ one can still show that
\,$\mu=\lim_{K\iin\,\scrK_{\aleph_0}^\circ} \mu\conv\lambda_K$ for
$\mu\iin\aMeasG$\,. But for other groups $G$ this is no longer true, since
$\scrK_{\aleph_0}^\circ$ (these are just the open compact subgroups)
is not sufficiently large. It may even happen (e.g. for $G=\Rgroup$) 
that $\scrK_{\aleph_0}^\circ$ is empty. 
\end{Rems}\vspace{1mm plus 2mm}

\section{Separation of strongly singular measures}
    \label{sec:strsingularSeparation}

In this section we extend Theorem \ref{th:perfectset} to the case of
strongly singular measures (Theorem~\ref{th:separationss}).
This will be used
in the proof of the general case of the Main Theorem.
Recall the Kakutani-Kodaira theorem (see \cite{Hu} Theorem 3 for a more
general version, a more special case\,--\,compactly generated
groups\,--\,which would be
sufficient for our purpose (when modifying slightly some arguments) is in
\cite[8.7]{Hewitt1963aha}).
\begin{lemma}
    \label{lemma:KKthm}
Let $H$ be a locally compact, $\sigma$-compact group.
If $U_n$, \,$n<\omega$, are neighbourhoods of $e_H$ then there exists
a compact normal subgroup $N$ of~$H$ such that $N\subseteq U_n$ for all $n$
and \,$H/N$ is metrizable.
\end{lemma}

\begin{lemma}
    \label{lemma:compactSubgroup}
If $G$ is any locally compact group then
$G$ has a compact subgroup $K_G$ such that $G/K_G$ is metrizable.
If $G$ is non-metrizable, then for any such group it follows that
\,$\chi(K_G) = \chi(G)$. If $G$ is $\sigma$-compact, $K_G$ can be chosen to
be normal.
\end{lemma}

\Proof
Take any symmetric compact neighbourhood $U$ of $e_G$\,.
Then \,$G_0  = \bigcup_{n=1}^\infty U^n$ is an open subgroup of $G$
(cf.~\cite[5.7]{Hewitt1963aha}), hence \,$\chi(G_0 )=\chi(G)$.
Since $G_0$ is compactly generated, by Lemma~\ref{lemma:KKthm} there is a
compact normal subgroup $K_G$ of $G_0$ such that $G_0 /K_G$ is metrizable
and thus $G/K_G$ is metrizable.
\\
If $\chi(G)>\aleph_0$, then $\chi(K_G)=\chi(G)$
since \,$\chi(G)=\max\bigl(\chi(G/ K_G), \chi(K_G)\bigr)$ by \linebreak
Lemma~\ref{lemma:basicproperties}.
\end{proof}

\begin{lemma}
    \label{lemma:gensaturation}
Let $G$ be a $\sigma$-compact, locally compact group and
$\nu,\nu'\iin\MeasG$,\linebreak $\nu,\nu'\geq 0$.
If $\nu\perp\nu'$ then there exist a~$\mathsf{K_\sigma}$-set
$E\subseteq G$ and a compact normal subgroup $N$ of $G$
such that $\nu(E)=\nu(G)$, $\nu'(EN)=0$ and $\chi(G/N)\leq\aleph_0$.
In particular, $\nu\conv\lambda_N\perp\nu'\conv\lambda_N$\,.\\
If $V$ is an $e_G$-neighbourhood in $G$,  we can find such an
$N$ with $N\subseteq V$\,.
\end{lemma}

\Proof
Since $\nu\perp\nu'$, there is a~$\mathsf{K_\sigma}$-set $E\subseteq G$
such that $\nu(E)=\nu(G)$ and $\nu'(E)=0$.
Thus there are compact sets $C_n$ such that \,$E=\bigcup_{n=1}^\infty C_n$\,.
For every $n$ and $m$ there exists an open set $W_{n,m} \supseteq C_n$
satisfying \,$\nu'(W_{n,m} )< 2^{-m}$.
Since $C_n$ are compact, there are neighbourhoods $U_{n,m}$ of $e_G$ with
\,$C_n U_{n,m} \subseteq W_{n,m}$\,.
By Lemma~\ref{lemma:KKthm} there is a compact normal subgroup $N$ of $G$ such
that
\,$N\subseteq\bigcap_{n=1}^\infty \bigcap_{m=1}^\infty U_{n,m}\cap V$
and $\chi(G/N)\leq\aleph_0$\,.
Then \,$\nu'(C_n N)\leq\nu'(C_n U_{n,m})\leq
\nu'(W_{n,m})< 2^{-m}$ for every $m$, giving
$\nu'(C_n N)=0$ for every $n$.
Since \,$EN = \bigcup_{n=1}^\infty C_n N$, it follows that
$\nu'(EN) =0$. Thus $\nu$ and $\nu\conv\lambda_N$ are
concentrated on $EN$\,, whereas
\,$\nu'\conv\lambda_N(EN)=\nu'(EN)=0$\,.
\end{proof}

\begin{lemma}
    \label{lemma:subsaturation}
Let $G$ be a $\sigma$-compact, locally compact group such that
$\chi(G)> \aleph_0$\,,
$K_1$ a normal subgroup of $G$ with \,$K_1\iin\ssubgr$\,, \,$V$ an
$e_G$-neigh\-bourhood,
and let\linebreak $\mu\iin\ssMeast{\chi(G)}(G)$ with $\mu\geq0$, $\mu(G)=1$.
Then there exists
a normal subgroup $K_2$ of $G$ with \,$K_2\iin\ssubgr$
and a~$\mathsf{K_\sigma}$-set $E\subseteq G$
such that \,$K_2\subseteq K_1\cap V$, \,$\mu(E)=1$\,,
\,$\mu\conv\lambda_{K_1}(EK_2)=0$ \,and $K_1/K_2$ is metrizable.
\end{lemma}

\Proof
Since $K_1\iin\ssubgr$\,, we have \,$\mu\perp\mu\conv\lambda_{K_1}$\,.
Let $N$ be as in Lemma~\ref{lemma:gensaturation}
with $\nu=\mu$ and $\nu'=\mu\conv\lambda_{K_1}$\,.
Put $K_2=K_1\cap N$\,. Recall (\cite{Hewitt1963aha}\;Thm.\,5.33) that
$K_1/K_2$ is topologically isomorphic to $K_1N/N$\,.
Since $\chi(G/N)\leq\aleph_0$\,, we get that $K_1/K_2$ is metrizable and
by \thetag{1} of Lemma~\ref{lemma:basicproperties} that $K_2 \iin\ssubgr$.
\end{proof}
\begin{corollary}
    \label{cor:onestep}
For $K_1,\,K_2,\,\mu$ as above, the following properties hold:
\begin{enumerate}
\item
\label{lemma:onestepsing:sing} \
$\mu\conv\lambda_{K_1} \perp \mu\conv\lambda_{K_2}$\,,
\item
\label{lemma:onestepsing:almostall} \
$\mu\conv K_2 \perp \mu\conv xK_2$ for $\lambda_{K_1}$-almost all
$x\iin K_1$\,,
\item
\label{lemma:onestepsing:alephzero} \
$\chi(K_1/K_2)=\aleph_0$ \,and \,$K_2\iin\ssubgr$\,.
\end{enumerate}
\end{corollary}\noindent
We use here the notations from the beginning of
Section~\ref{sec:factorization}.

\Proof
We have
\[
\int_{K_1} \mu\conv x(EK_2)\,d\lambda_{K_1} (x) =
\int_{K_1} \mu(EK_2x^{-1})\,d\lambda_{K_1} (x)
=\mu\conv\lambda_{K_1} (EK_2)=0
\]
and therefore \,$\mu\conv x(EK_2)=0$ for $\lambda_{K_1}$-almost all
$x\iin K_1$\,.
\\
Take any $x\iin K_1$ such that \,$\mu\conv x(EK_2)=0$ and any
$y,y'\iin K_2$\,. Then
\[
\mu\conv x y(Ey') =  \mu\conv x \conv y(Ey')
= \mu\conv x (Ey' y^{-1}) \leq \mu\conv x (EK_2) = 0 \; ,
\]
while \,$\mu\conv y'(Ey')=\mu(E)=1$
which shows that \,$\mu\conv y' \perp \mu\conv xy$\,. Thus
\,$\mu\conv K_2 \perp \mu\conv xK_2$\,.
If $K_2$ were open in $K_1$\,, then \,$\lambda_{K_1}(K_2)>0$ and we would
get \,$\mu\conv K_2 \perp \mu\conv K_2$\,. Thus $K_1/K_2$ must be non-discrete,
giving \,$\chi(K_1/K_2)=\aleph_0$\,.
\end{proof}\vspace{1.5mm}

As usual, a cardinal number $\tau$ will be identified with the minimal
ordinal number~$\alpha$ for which $\card{\alpha}=\tau$\,.\vspace{-1.5mm}

\begin{lemma}
    \label{lemma:induction}
Let $G$ be a $\sigma$-compact, locally compact group, $K_0$ a compact
normal subgroup such that \,$\chi(G)>\chi(G/K_0)+\aleph_0$
and $\mu\iin\ssMeast{\chi(G)}(G)$ with $\mu\geq0$, $\mu(G)=1$.
Then for \,$0<\alpha\le\chi(G)$ there exist compact normal
subgroups $K_\alpha$ of $G$ and $x_\alpha\iin K_\alpha$ such that\vspace{-1mm}
\begin{itemize}
\item[(i)] \
$K_\alpha\supseteq K_{\alpha+1}$ and
\,$\chi(K_{\alpha}/K_{\alpha+1})=\aleph_0$ for all $\alpha<\chi(G)$;
\item[(ii)] \
$K_\beta = \bigcap_{\alpha<\beta} K_\alpha$ for all limit ordinals
\,$\beta\le\chi(G)$;
\item[(iii)] \
$ K_{\chi(G)}=\{e_G\}$;
\item[(iv)] \
$\mu\conv K_{\alpha+1}\perp\mu\conv x_\alpha K_{\alpha+1}$ \;and
\;$\mu\conv\lambda_{K_\alpha} \perp \mu\conv\lambda_{K_{\alpha+1}}$
\,for all $\alpha<\chi(G)$.
\end{itemize}
\end{lemma}\noindent
Note that by Corollary\;\ref{corollary:basicproperties:monotone} it
follows that \,$\chi(G/K_\alpha) = \card{\alpha}+\chi(G/K_0)+\aleph_0$
\,for \,$0<\alpha\le\chi(G)$\,, in particular \,$K_\alpha \iin \ssubgr$
\,for \,$\alpha<\chi(G)$, and $K_\alpha$ is not open in $G$ \,for
$\alpha>0$.

\Proof
Let $(V_{\alpha})_{\alpha<\chi(G)}$ be a family of $e_G$-neighbourhoods
satisfying \,$\bigcap_{\alpha<\chi(G)} V_\alpha=\{e_G\}$.
Now use transfinite induction, applying Lemma~\ref{lemma:subsaturation} and
Corollary~\ref{cor:onestep} with
the requirement \,$K_{\alpha+1}\subseteq V_\alpha$\,.
\end{proof}\vspace{1mm}

\begin{theorem}
    \label{th:separationss}
Let $G$ be a locally compact group, \,$\aleph_0\le\tau\le\chi(G)$
and \,$\mu\iin\ssMeast{\tau}(G)$.
There exists a set $P\subseteq G$ such that \,$\card{P}=2^\tau$
and \,$\mu\conv p \perp \mu\conv p'$ for all $p,p' \iin P$,
$p\neq p'$. Thus $\mu$ is \,$2^\tau$-\,thin.\vspace{-1mm}
\end{theorem}\noindent
If $G$ is non-metrizable, we can find such a $P$ satisfying $P\subseteq K_G$
 \,(defined by Lemma\,\ref{lemma:compactSubgroup}).

\Proof
We have $\mu\iin\ssMeas(G,K)$ for some compact subgroup $K$ with
$\chi(G/K)=\tau$.
The case $\tau=\aleph_0$ was settled in Theorem \ref{th:perfectset} \,(recall
that $\ssMeas(G,K)\subseteq\sMeas(G)$\,). Thus, we may assume that
$\tau>\aleph_0$ and, replacing $G$ by some open subgroup (using
Lemma\;\ref{lemma:heredit}), we may assume that
$G$ is $\sigma$-compact. Then by \thetag{3}
of Lemma~\ref{lemma:basicproperties}, $K$ can be replaced by a normal
subgroup. In this way (using Corollary \ref{corollary:qIdentification}),
the proof is reduced to the case where
\,$\mu\iin\ssMeast{\chi(G)}(G)$ and $G$ is $\sigma$-compact and non-metrizable
(i.e., $\tau=\chi(G)>\aleph_0$).
Without loss of generality, assume that \,$\mu\geq0$, $\mu(G)=1$.
Recall the notational conventions in the proof of Theorem~\ref{th:perfectset}.

Let $K_\alpha$ and \,$x_\alpha\iin K_\alpha$ for $\alpha<\chi(G)$ be
as in Lemma~\ref{lemma:induction}, with $K_0=K_G$\, (obtained from
Lemma\;\ref{lemma:compactSubgroup}).
Induction on $\beta \le\chi(G)$ will be used to construct elements
$d_*\iin K_0$ for all $d\subseteq \beta$
such that
\begin{itemize}
\item[$(1^{\bullet})$]
if  \,$\alpha<\beta\le\chi(G)$ and $d\subseteq\beta$
\;then \,$d_*\iin(d\cap\alpha)_*K_\alpha $~;
\item[$(2^{\bullet})$]
if $\beta\le\chi(G)$ and $d$ and $d'$ are distinct subsets of  $\beta$
then
\,$\mu\conv d'_*K_\beta \perp \mu\conv d_*K_\beta$~.
\end{itemize}
That will conclude the proof,  since in view of $(2^{\bullet})$
we can take \,$P=\{d_*\!: d\subseteq\chi(G) \}$.\vspace{1mm}

To carry out the induction, start with $\beta = 0$ and define
$\emptyset_* = e_G$\,.
Now assume that $0\neq\beta<\chi(G)$ and that $d_*$ have been defined for
all $d \subseteq\alpha$ and $\alpha<\beta$.
If  $d \subseteq\beta$ and $\beta$ is a limit ordinal define
$d_*$ to be any point in
\,$\bigcap_{\alpha<\beta }(d\cap \alpha)_*K_\alpha$ \ (this is a decreasing
family of sets by $(1^{\bullet})$ and since $(K_\alpha)$ is decreasing)\,.
If $\beta=\alpha+1$ then define\vspace{-2mm}
\[
d_* = \left\{ \begin{array}{ll}
                 \phantom{x}(d \cap \alpha)_* & \mathrm{when}\ \
		 	\alpha\notin d                 \\
                 x_\alpha (d \cap \alpha)_* & \mathrm{when}\ \
		 	\alpha\in d\ .
              \end{array}
      \right.
\]

Slightly informally, this can be described as follows. If $\beta=\gamma+n$\,,
where $n$ is finite, $d\subseteq\beta$ \,and
$\beta\setminus\gamma=\{\gamma_0,\dots,\gamma_l\}$, where
$\gamma\le \gamma_0<\dots<\gamma_l<\beta$\,, then
$d_*=x_{\gamma_l}\cdots x_{\gamma_0}(d\cap\gamma)_*$
\;(for technical reasons,
we revert here the order of the factors, compared to the proof of
Theorem \ref{th:perfectset}). For  limit ordinals $\beta$ there is some choice
in the definition of $d_*$\,, since the net of ``partial products"
$\bigl(\,(d\cap\alpha)_*\bigr)_{\alpha<\beta}$ will converge in general only
in the quotient $G/K_\beta$ (being the projective limit of
$(G/K_\alpha)_{\alpha<\beta}$\,; compare the proof of
Corollary \ref{corollary:basicproperties:monotone}).

Property $(1^{\bullet})$ is easy to prove from the
definition, using normality of $K_\alpha$ in $G$\,.
To prove $(2^{\bullet})$, if  $d,d'\subseteq\beta$ with $d\neq d'$, take
$\beta'\le\beta$ minimal such that \,$d\cap\beta'\neq d'\cap\beta'$.
In view of $(1^{\bullet})$, replacing $\beta,d,d'$ by
$\beta',d\cap\beta',d'\cap\beta'$,
it is enough to consider the case where $\beta=\alpha+1$
and $d, d'\subseteq\beta$ are such that
\,$d'\cap \alpha = d\cap \alpha$\,,
$d_\ast = (d\cap\alpha)_\ast$ and
$d'_\ast = x_\alpha (d\cap\alpha)_\ast = x_\alpha d_\ast$\,.
Since \,$\mu\conv K_{\alpha+1}\perp\mu\conv x_\alpha K_{\alpha+1}$, we have
also
\,$\mu\conv K_{\alpha+1} d_\ast \perp\mu\conv x_\alpha K_{\alpha+1} d_\ast$\,.
As $K_{\alpha+1}$ is a normal subgroup of $G$\,,\vspace{-1.5mm} 
\begin{alignat*}{2}
K_{\alpha+1} d_\ast\ &=\  &d_\ast  K_{\alpha+1}\ &=\ d_\ast K_\beta  \\
x_\alpha K_{\alpha+1} d_\ast\ &=\  x_\alpha &d_\ast K_{\alpha+1}
\ &=\ d'_\ast K_\beta\ . 
\end{alignat*}
Thus \,$ \mu\conv d_\ast K_\beta\,\perp\,\mu\conv d'_\ast K_\beta$~.
\end{proof}

\begin{corollary}
    \label{cor:thinss}
Let $G$ be a locally compact group, $K$ a non-open compact subgroup. Then
every \,$\mu\in\ssMeas(G,K)$ is $\card{G/K}$\,-\,thin.
\end{corollary}

\Proof
Recall that \,$\card{G/K}=\kappa(G)\cdot 2^{\chi(G/K)}$ holds when
$K$ has infinite index in $G$. By Theorem \ref{th:perfectset}, 
$\mu$ is $\kappa(G)$\,-\,thin and by Theorem \ref{th:separationss}
(using that \,$\ssMeas(G,K)\!\subseteq\ssMeast{\chi(G/K)}(G)$\,),
$\mu$ is $2^{\chi(G/K)}$\,-\,thin when $K$ is not open. This covers all
possible cases.
\end{proof}

\begin{Rems} We mention here some further results that can be shown by
similar methods (this is not needed for the proof of the Main Theorem).
\item{(a)} \,There are some converse statements to Corollary \ref{cor:thinss}.
For $\mu\iin\MeasG,\ \mu\neq0$,\linebreak  put
\,$K_\mu=\{x\iin G\!:\mu\conv x=\mu\,\}$ \,and
\,$M_\mu=\{\mu\conv x\!:x\iin G\,\}$.
Then $K_\mu$ is the maximal compact subgroup of $G$ such that
\,$\mu\iin\Meas(G/K_\mu)$.
Clearly, we have
\,$d(M_\mu)\le\linebreak \lvert M_\mu\rvert=\card{G/K_\mu}$\,.
In particular, $\mu$ cannot be $\tau$-thin for $\tau>\card{G/K_\mu}$.
\,Hence the value $\card{G/K}$ in Corollary \ref{cor:thinss} is the best
possible when $\mu\ne0$\,.

The equality $d(M_\mu)=\card{G/K_\mu}$ holds iff either $M_\mu$ is finite
(e.g. when $G$ is finite) or $\kappa(G)\ge2^{\chi(G/K_\mu)}$
or $\mu\notin\aiMeast{\tau}(G)$ where $\tau$ denotes the least cardinal
such that $2^\tau=2^{\chi(G/K_\mu)}$ \;(if the first two options do not hold
then $K_\mu$ must be non-open). In the remaining cases, $d(M_\mu)$ can be
expressed similarly, using the decomposition of $\sMeasG$ described at the
end of Section \ref{sec:compsub}. For $\mu\notin\aMeasG$ (which implies that
$K_\mu$ is not open),
$\mu$ is $d(M_\mu)$-thin iff either
\,$\kappa(G)\ge2^{\chi(G/K_\mu)}$ or \,$\mu\perp\aiMeast{\tau}(G)$
\,(with $\tau$~as above).
In all these cases $d(M_\mu)=\card{G/K_\mu}$ holds (but not conversely).

Assuming e.g. the generalized continuum hypothesis, it follows that if $K$ is
a non-open compact subgroup with $\kappa(G)<2^{\chi(G/K)}$
and \,$\mu\iin\Meas(G/K)$
is $\card{G/K}$\,-\,thin, then $\mu$ must be strongly singular.
\;On the other hand, under the assumption $2^{\aleph_1} = 2^{\aleph_0}$
(which is also consistent with ZFC), one gets that for a group $G$ with
$\chi(G)\le\aleph_1$ all $\mu\iin\sMeas(G)$ are $\card{G}$\,-\,thin.
\vspace{0mm plus .5mm}
\item{(b)} Another generalization is to consider a locally compact space
$\Omega$ with a (jointly continuous) left action of a locally compact group
$G$ (see \cite{Jew} and \cite{Larsen}). This induces a left action of
$\MeasG$ on $\Meas(\Omega)$ which will be written again as $\mu\conv\nu$ for
\linebreak $\mu\iin\MeasG,\;\nu\iin\Meas(\Omega)$. There need not exist a
non-zero $G$-invariant measure on~$\Omega$ (and if one exists, it need not be
unique), but one can define e.g.
$\aMeas(\Omega)=\linebreak\aMeasG\conv\Meas(\Omega)$. If one uses
now {\it left} translations by elements of $G$, most of the constructions in
Sections \ref{sec:singularSeparation} and \ref{sec:strsingularSeparation}
work again. For example, every $\mu\iin\sMeas(\Omega)\ (=\aMeas(\Omega)^\perp)$
is $2^{\aleph_0}$\,-\,thin (with respect to left translations from $G$).
Similarly, one can define strongly singular measures on $\Omega$ and there
is an analogue of Theorem \ref{th:separationss}. In these cases one uses
translations by elements from some compact subset of $G$. The statements
based on non-compactness (i.e., using $\kappa(G)$\,) do not always extend
to this setting, depending on further properties of the action.
\end{Rems}\vspace{.5mm plus .5mm}

\section{Proof of the Main Theorem --- general case}
    \label{sec:caseII}

First we show some further properties of the spaces $\ssMeas(G,K)$ and
$\aiMeas(G,K)$. Then (Theorem \ref{th:caseII}) we arrive at the
inductive argument to prove the Main Theorem.

\begin{lemma}
    \label{lemma:singular}
Let $G$ be a locally compact group,
$K$ a compact subgroup of $G$,\linebreak $\mu\iin\MeasG$. We have
\,$\mu\perp\Meas(G/K)$ \;iff \ $\mu\perp\abs{\mu}\conv\lambda_K$\,.\\
If \;$\nu\iin\Meas(G/K)$, \,$\nu\perp\abs{\mu}\conv\lambda_K$\,, then
\,$\nu\perp\mu$\,.
\end{lemma}

\Proof
It is enough to consider $\mu\ge0$ and $\nu\ge0$.
For clarity, we take up the more precise notation of
Lemma\;\ref{Identify}.
For the second statement, assume that $\nu\iin\Meas(G,K)$ and
$\nu\perp\mu\conv\lambda_K$\,.
Then for the image measures,
$\dot\nu\perp(\mu\star\lambda_K)^{\textstyle\cdot}$
holds as well (by Lemma\;\ref{Identify}). Hence there is a Borel set
$E_0$ in $G/K$ such that \,$\dot\nu(E_0)=\dot\nu(G/K)$ and
$(\mu\star\lambda_K)^{\textstyle\cdot}(E_0)=0$. Recall that
$(\mu\star\lambda_K)^{\textstyle\cdot}=\dot\mu$\,.
Put $E=\pi^{-1}(E_0)$. Then $E$ is a Borel set in $G$ and, by the definition
of image measures, we get $\nu(E)=\nu(G)$ and $\mu(E)=0$, so that
\,$\nu\perp\mu$\,.

In the first part, one direction is clear. For the other one, assume that
$\mu\perp\mu\conv\lambda_K$ and take any $\nu\iin\Meas(G,K)$ ($\nu\ge0$).
The Lebesgue decomposition of $\nu$ with respect to
$\mu\conv\lambda_K\iin\Meas(G,K)$
yields \,$\nu=\nu_0+\nu_1$ where $\nu_0,\nu_1\iin\Meas(G,K)$,
$\nu_0,\nu_1\geq 0$, \,
$\nu_0\perp\mu\conv\lambda_K$ and $\nu_1\ll\mu\conv\lambda_K$\,.
The second statement of the lemma implies \,$\nu_0\perp\mu$ and from
$\nu_1\ll\mu\conv\lambda_K$ and \,$\mu\perp\mu\conv\lambda_K$ we get
\,$\nu_1\perp\mu$\,. Hence \,$\nu\perp\mu$\,.
\end{proof}\vspace{1mm}

If $G$ is a locally compact group,
$K$ a compact subgroup of $G$ and $G/K$ non-metrizable, put
\[\scrK_K= \scrK_K (G)=
\{L\supseteq K\!: \,L\text{ compact subgroup of }G,\ \chi(G/L)<\chi(G/K)\,\}
\;.\]
By \thetag{2} in Lemma\;\ref{lemma:basicproperties}
the family $\scrK_K$ is downwards directed by inclusion. For\linebreak
$L\iin\scrK_K$\,, recall that (using the identifications of
Lemma\;\ref{Identify}) \,$\Meas(G/L)$
is a subalgebra of $\Meas(G/K)$\,.  The next Lemma will provide a more
intrinsic description of the subspaces $\aiMeas(G,K)$ and $\ssMeas(G,K)$ of
$\Meas(G/K)$.

\begin{lemma}
    \label{lemma:approx}
Let $G$ be a locally compact group, $K$ a compact subgroup of $G$ such that
$G/K$ is non-metrizable, and $\nu\iin\MeasG$. We have $\nu\iin\aiMeas(G,K)$
iff \,$\nu=\lim_{L\in\scrK_K} \nu\conv\lambda_L $ (norm limit). This is
also equivalent to \,$\inf_{L\in\scrK_K}\norm{\nu-\nu\conv\lambda_L}=0$\,.
\\
For $\nu\iin\Meas(G/K)$ we have $\nu\iin\ssMeas(G,K)$ iff
\,$\nu\perp\abs{\nu}\conv\lambda_L$ for all $L\iin\scrK_K$\,.
\end{lemma}\vspace{-1mm}\noindent
In particular, $\aiMeas(G,K)$ coincides with the norm-closure of
\;$\bigcup_{L\iin\scrK_K}\Meas(G/L)$\,. Moreover, it follows easily that
$\delta_x\conv\lambda_K$ (or more generally, every $\mu\iin\Meas(G/K)$ with
$\mu\ll\delta_x\conv\lambda_{K'}$ for some compact subgroup $K'$ of $G$
with $\chi(G/K')=\chi(G/K)$\,) belongs to $\ssMeas(G,K)$ for all $x\iin G$\,.
\\
An important step in the proof of the Main Theorem will be to show that
the limit condition for $\aiMeas$ can be weakened
(Proposition\;\ref{ssnolimit}): \ if
the net \,$(\nu\conv\lambda_L)_{L\in\scrK_K}$ converges in the
{\it \wstar topology} of $\MeasG^{\ast\ast}$ (to any limit),
it already follows that $\nu\in\aiMeas(G,K)$.

\Proof
Put $\tau=\chi(G/K)$. If $L,L'$ are compact subgroups of $G$ with
$L\subseteq L'$, then $\lambda_{L'}\conv\lambda_L=\lambda_{L'}$\,. Thus
\,$\norm{\nu\conv\lambda_{L'}-\nu\conv\lambda_L}\le
\norm{\nu\conv\lambda_{L'}-\nu}$
and it follows that\linebreak
$\norm{\nu\conv\lambda_L-\nu}\le2\,\norm{\nu\conv\lambda_{L'}-\nu}$\,.
This shows that \,$\nu=\lim_{L\in\scrK_K} \nu\conv\lambda_L $ is
equivalent to $\inf_{L\in\scrK_K}\norm{\nu-\nu\conv\lambda_L}=0$ and
this equivalence persists when $\scrK_K$ is replaced by any family
$\scrD$ of compact subgroups which is downwards directed.

Since \,$\scrK_K\subseteq\scrK_\tau^\circ$\,, it follows that
\,$\nu=\lim_{L\in\scrK_K} \nu\conv\lambda_L$ implies
\,$\nu=\lim_{L\in\scrK_\tau^\circ} \nu\conv\lambda_L$\,, i.e.
$\nu\iin\aiMeast{\tau}(G)$\,. Furthermore,
\,$\nu\conv\lambda_L\iin\Meas(G/K)$ for $L\iin\scrK_K$ implies
\,$\nu\iin\Meas(G/K)$\,.

For the converse, if $L\iin\scrK_\tau^\circ$\,, let $H$ be an open
$\sigma$-compact subgroup of $G$ containing $K$ and $L$\,. By
Corollary~\ref{corollary:basicproperties:classes} there exists
$L'\iin\scrK_\tau^\circ$
such that $L'$ is normal in $H$ and $L'\subseteq L$\,. Then $KL'$ is a group,
by \thetag{1} of Lemma~\ref{lemma:basicproperties} we have $KL'\in\scrK_K$
and for \,$\nu\iin\Meas(G/K)$ we get (using
$\lambda_{KL'}=\lambda_K\conv\lambda_{L'}$)
that $\nu\conv\lambda_{KL'}=\nu\conv\lambda_{L'}$\,. Thus
\,$\nu=\lim_{L\in\scrK_\tau^\circ} \nu\conv\lambda_L $ \,implies
\,$\nu=\lim_{L\in\scrK_K} \nu\conv\lambda_L $\,.

For the second part, we may assume that $\nu\ge0$. By
Theorem \ref{th:measureclasses}, we have $\nu=\nu_1+\nu_2$\,, where
$\nu_1\iin\ssMeas(G,K),\,\nu_2\iin\aiMeas(G,K)$\,. If $\nu_2\neq0$\, we
have by the first part $\nu_2\not\perp\nu_2\conv\lambda_L$ for some
$L\iin\scrK_K$\,. Since \,$0\le\nu_2\le\nu$\,, this implies
\,$\nu\not\perp\nu\conv\lambda_L$\,. The other direction follows
immediately from the definition of $\ssMeas(G,K)$.
\end{proof}

\begin{corollary}
    \label{corollary:inbetween}
Let $G$ be a locally compact group, $K$ a compact subgroup of $G$ such that
$G/K$ is non-metrizable, and \,$\nu\iin\MeasG$.
Let \,$\scrD\subseteq \scrK_{\chi(G/K)}^\circ$ be a family of compact
subgroups of $G$ that is cofinal with $\scrK_K$ \,(i.e., for $L\iin\scrK_K$
there exists $L'\iin\scrD$ with $L'\subseteq L$).
We have \,$\nu\iin\aiMeas(G,K)$ iff 
\ $\inf_{L\in\scrD}\norm{\nu-\nu\conv\lambda_L}=0$\,.
For $\nu\iin\Meas(G/K)$ we have \,$\nu\iin\ssMeas(G,K)$ iff
\;$\nu\perp\abs{\nu}\conv\lambda_L$ for all $L\iin\scrD$\,.
\\
Assuming in addition that $\scrD$ is downwards directed, it follows that
\,$\nu\!\iin\!\aiMeas(G,K)$ \,iff
\,$\nu=\lim_{L\in\scrD} \nu\conv\lambda_L $ (norm limit).
\end{corollary}

\Proof
The first part and the final statement follow immediately from
Lemma\;\ref{lemma:approx} and its
proof. For the second part, observe that by Lemma\;\ref{lemma:singular},
$\,\nu\perp\abs\nu\conv\lambda_{L'}$ implies
$\,\nu\perp\abs\nu\conv\lambda_L$ for all $L\supseteq L'$ \,(since
$\Meas(G/L)\subseteq\Meas(G/L')$\,).
\end{proof}

\begin{lemma}
    \label{lemma:linapprox}
Let $G$ be a locally compact group, $K$ a compact subgroup of $G$ such that
$G/K$ is non-metrizable. Put \,$\tau=\chi(G/K)$ and let
\,$\scrD=\{K_\alpha\!:\alpha<\tau\}$ be a subfamily of $\scrK_K$ such that
\,$K_\alpha\supseteq K_\beta$ for $\alpha<\beta<\tau$\,, and
\,$K = \bigcap_{\alpha<\tau} K_\alpha$\,. If $\tau$ is a successor
cardinal \;(i.e., $\tau$ is not the supremum
of the family of smaller cardinals), then $\scrD$ is cofinal in $\scrK_K$\,.
\end{lemma}

\Proof
The argument is similar
as in the proof of Corollary \ref{corollary:basicproperties:monotone}.
The family of sets
$VK_\alpha$ where $V$ is an $e_G$-neighbourhood and $\alpha\!<\!\tau$ defines
a neighbourhood basis in $G/K$\,.
Given $L\iin\scrK_K$\,, it follows that there exists a family
$(V_i)_{i\in I}$ of
$e_G$-neighbour\-hoods and $\alpha_i<\tau$ such that
\,$\bigcap_{i\in I}\,V_iK_{\alpha_i}\subseteq L$\,,
where \,$\card I=\chi(G/L)<\tau$\,. Then the assumptions about $\tau$
imply that \,$\beta=\sup\{\,\alpha_i\!: i\in I\}<\tau$ \,and monotonicity
gives \,$K_\beta\subseteq L$\,.
\end{proof}

\begin{Rem}
The same proof works for limit cardinals $\tau$ that are regular  (i.e.,
if $\tau$ cannot be expressed
as the supremum of a set of cardinality less than $\tau$, whose elements
are cardinals less than $\tau$). Note that in these cases, it follows by
Lemma\;\ref{lemma:approx} that
$\nu=\lim_{\alpha<\tau} \nu\conv\lambda_{K_\alpha} $ (norm limit) holds for
all \,$\nu\iin\aiMeas(G,K)$ and $(K_\alpha)$ as above. One might expect this
to be valid without restrictions on $\tau$\,, but the following
example shows a different behaviour. Let $F$ be
a non-trivial finite group and consider a product group $G=F^\tau$ for an
infinite cardinal $\tau$ and $K=\{e\}$. For infinite $\alpha<\tau$ take
$K_\alpha=F^{\tau\setminus\alpha}$ (embedded into $G$ in the usual way).
For a non-empty subset $I$ of $\tau$ with $\aleph_0\le\card I<\tau$ take
$L=F^{\tau\setminus I}$.
Then $G/K_\alpha\cong F^\alpha$, $G/L\cong F^I$, $\chi(G/K)=\chi(G)=\tau$\,,
$\chi(G/K_\alpha)=\abs\alpha$, $\chi(G/L)=\card I$\,. Thus
$K_\alpha,L\iin\scrK_K$ and clearly \,$\bigcap_\alpha K_\alpha=K$\,. If $I$
can be chosen so that \,$\sup I=\tau$ \,(this works for any limit cardinal
that is not regular), then $K_\alpha\nsubseteq L$ for all $\alpha$\,, hence
$\{K_\alpha\}$ is not cofinal in $\scrK_K$\,. Furthermore, taking
$\nu =\lambda_L\in\aiMeas(G,K)\!=\!\aiMeast{\tau}(G)$, it is easy to see that 
\,$\nu\conv\lambda_{K_\alpha}=\lambda_{LK_\alpha}\perp\nu$\,
for all $\alpha$ \;(since $I\setminus\alpha$ must be infinite), hence
$(\nu\conv\lambda_{K_\alpha})$ does not converge to $\nu$ in norm  (but, in
fact one can see as in Corollary\,\ref{cor:groupapprox}\,\thetag{2} below that
no other limit is possible).
\end{Rem}

\begin{lemma}
    \label{lemma:decomposition}
Let $G$ be a locally compact group, $K$ a compact subgroup of $G$ such that
$G/K$ is non-metrizable, $L\iin\scrK_K$ and $\mu\iin\ssMeas(G,K)$.
Put $\tau=\chi(G/K)$.
There exists a family \,$\scrD=\{K_\alpha:\alpha<\tau\}\subseteq\scrK_K$
such that \,$L\supseteq K_\alpha\supseteq K_\beta$ and
\,$\mu\conv\lambda_{K_\alpha} \perp \mu\conv\lambda_{K_\beta}$ for
$\alpha<\beta<\tau$\,, $K_\beta = \bigcap_{\alpha<\beta} K_\alpha$ for all
limit ordinals \,$\beta<\tau$\,, $\chi(G/K_\alpha)=\chi(G/L)+\abs{\alpha}$
when $\alpha<\tau$ is infinite, and \,$K = \bigcap_{\alpha<\tau} K_\alpha$\,.
\end{lemma}

\noindent
Observe that in Lemma\;\ref{lemma:compactSubgroup} the group $K_G$ can
always be chosen to be normalized by a given compact subgroup $K$ of $G$\,.
Then $L=K_GK$ can be used for Lemma\;\ref{lemma:decomposition} and then the
corresponding family  $\scrD$ satisfies $\chi(G/K_\alpha)=\abs{\alpha}$ when
$\alpha$ is infinite.

\Proof
Consider an open $\sigma$-compact subgroup $H\supseteq L$ that contains the
support of~$\mu$\,. By
Corollary\;\ref{corollary:basicproperties:classes} and
Lemma\;\ref{lemma:heredit}, we may replace
$G$ by $H$ and assume that $G$ is $\sigma$-compact. Replacing $\mu$ by
$\abs\mu$\,, we assume $\mu\ge0$ \,(note that if
$\abs\mu\conv\lambda_{K_\alpha} \perp \abs\mu\conv\lambda_{K_\beta}$
then $\mu\conv\lambda_{K_\alpha} \perp  \mu\conv\lambda_{K_\beta}$ and
use Theorem\;\ref{th:measureclasses}). Adding $\lambda_K$ if necessary,
we can assume $\mu\neq0$\,.

Put $N=\bigcap_{x\in G}\,xKx^{-1}$, $N_L\!=\bigcap_{x\in G}\,xLx^{-1}$,
then $N\subseteq K\cap L$\,, both $N,N_L$ are normal in~$G$\,,
$\chi(G/N)=\tau$\,, $\chi(G/N_L)=\chi(G/L)$
\,(Lemma\;\ref{lemma:basicproperties}) and
$\mu\iin\ssMeas(G,N)$ \,(Lemma\;\ref{lemma:heredit}).
Corollary\;\ref{corollary:qIdentification} gives
$\mu\iin\ssMeast{\tau}(G/N)$.
We put $\widetilde K_0=N_L/N$ and apply Lemma\;\ref{lemma:induction} to
\,$\widetilde G=G/N$ and $\mu$\,. This produces a family of compact normal
subgroups $\widetilde K_\alpha$ of  $\widetilde G$ \,($\alpha<\tau$) such that
$\widetilde K_\alpha\supseteq \widetilde K_{\alpha+1}$\,,
$\chi(\widetilde K_{\alpha}/\widetilde K_{\alpha+1})=\aleph_0$ \,and
\,$\mu\conv\lambda_{\widetilde K_\alpha} \perp
\mu\conv\lambda_{\widetilde K_{\alpha+1}}$ for
$\alpha<\tau$\,,
$\widetilde K_\beta = \bigcap_{\alpha<\beta} \widetilde K_\alpha$ for all
limit ordinals \,$\beta<\tau$ and
\,$\bigcap_{\alpha<\tau} \widetilde K_\alpha=\{N\}$\,.
By the second part of Lemma\;\ref{lemma:singular} (applied for
$\widetilde K_{\alpha+1}$ and
$\nu=\mu\conv\lambda_{\widetilde K_\alpha}$), this implies
$\mu\conv\lambda_{\widetilde K_\alpha} \perp
\mu\conv\lambda_{\widetilde K_\beta}$ for $\alpha<\beta<\tau$
\,(we have $\mu\conv\lambda_{\widetilde K_\beta}\conv
\lambda_{\widetilde K_{\alpha+1}}=\mu\conv\lambda_{\widetilde K_{\alpha+1}}$%
).
Then $\widetilde K_\alpha=K'_\alpha/N$ for compact normal subgroups
$K'_\alpha$ of $G$\,.

Finally, put $K_\alpha=K'_\alpha K$\,. By normality, these are compact
subgroups. Observe that $\mu\iin\Meas(G/K)$
implies $\mu=\mu\conv\lambda_K$\,, hence using the identifications of
Lemma\;\ref{Identify} we have \,$\mu\conv\lambda_{\widetilde K_\alpha}=
\mu\conv\lambda_{K'_\alpha}=\mu\conv\lambda_K\conv\lambda_{K'_\alpha}=
\mu\conv\lambda_{K_\alpha}$\,. Thus $\mu\conv\lambda_{K_\alpha}\perp
\mu\conv\lambda_{K_{\alpha+1}}(\neq0)$ which implies that $K_{\alpha+1}$
is not open in $K_\alpha$ \,(this need not be true, if one does this
construction without~$\mu$\,, resp. $\mu=0$, and then one would have to
remove repetitions, passing to some subfamily). Furthermore, since
$K'_\alpha/K'_{\alpha+1}\cong\widetilde K_\alpha/\widetilde K_{\alpha+1}$
\,and $K_\alpha/K_{\alpha+1}$ is homeomorphic to
$K'_\alpha\big/\bigl(K'_{\alpha+1}(K\cap K'_\alpha)\bigr)$\,, it follows that
$\chi(K_{\alpha}/K_{\alpha+1})=\aleph_0$ for\linebreak $\alpha<\tau$\,.
It is easy to see that \,$K_\beta = \bigcap_{\alpha<\beta} K_\alpha$ holds for
limit ordinals \,$\beta<\tau$ and
\,$\bigcap_{\alpha<\tau} K_\alpha=K$\,.
Lemma\;\ref{lemma:basicproperties} implies that
$\chi(G/K_\alpha)=\chi(G/K_0)+\abs{\alpha}$ when
$\alpha$ is infinite. Since $K_0=N_LK\subseteq L$\,, we have
$\chi(G/K_0)=\chi(G/L)$ and it follows that the the $K_\alpha$ will
satisfy our demands.
\end{proof}

\begin{corollary}
	\label{cor:groupapprox}
Let $G$ be a locally compact group, $K$ a compact subgroup of $G$ such that
$G/K$ is non-metrizable.
\item[(1)] \;We have $\bigcap_{L\in\scrK_K}L=K$ \,and
for every $L\iin\scrK_K$ and every infinite cardinal $\tau_1$ with
$\chi(G/L)\le\tau_1<\chi(G/K)$ there exists $L'\iin\scrK_K$
with  $\chi(G/L')=\tau_1$\,, $L'\subseteq L$\,.
\item[(2)] \;For $f\iin C_0(G/K)$ we have
\;$\lim_{L\in\scrK_K} \lambda_L\odot f=f$
in the norm topology. It follows that for
$\mu\iin\Meas(G/K)$, \;$\lim_{L\in\scrK_K} \mu\conv\lambda_L=\mu$
holds in the \wstar topology of $\MeasG$\,, i.e.,
$\lim_{L\in\scrK_K} \langle\,\mu\conv\lambda_L\,,\,f\,\rangle=
\langle\mu,f\rangle$ \,for $f\in C_0(G)$.
\end{corollary}\noindent
Since the embedding of $C_0(G/K)$ (mentioned after Lemma\,\ref{Identify}) is
isometric one can use in \thetag{2} either $\lVert\cdot\rVert_\infty$ of
$C_0(G/K)$ or that of $C_0(G)$. Corresponding limit relations as in \thetag{2}
hold for any downwards directed family $\scrD$ of compact subgroups of $G$
satisfying \,$\bigcap\scrD=K$\,.

\Proof
Part\,\thetag{1} follows immediately from
Lemma\;\ref{lemma:decomposition}. If $V$ is a neighbourhood
of $K$ in $G$\,, it follows by compactness that $L\subseteq V$ holds for some
$L\iin\scrK_K$ and (by uniform continuity) this implies convergence of
\,$(\lambda_L\odot f)_{L\in\scrK_K}$ for
$f\in C_0(G/K)$. Then the last statement follows by duality
(recall that $\langle\mu,f\rangle=\langle\mu,\lambda_K\odot f\rangle$ and
$\lambda_K\odot f\iin C_0(G/K)$ for $\mu\iin\Meas(G/K),\:f\in C_0(G)$\,).
\end{proof}

\begin{lemma}
    \label{lemma:cofinalsets}
Let $G$ be a locally compact group, $K$ a compact subgroup of $G$ such that
$G/K$ is non-metrizable, and $\mu\iin\ssMeas(G,K)$. There exist two cofinal
subsets $\scrC$ and $\scrC'$ of $\scrK_K$
such that \,$\mu\conv\lambda_L \perp \mu\conv\lambda_{L'}$
for all $L\iin\scrC$, $L'\iin\scrC'$.
\end{lemma}

\Proof
Put $\tau=\chi(G/K)$. For the reasons explained in the remark after
Lemma\;\ref{lemma:linapprox}, we separate two cases.\\
\textit{Case I:}
\,$\tau$ is a successor cardinal.
Let $\scrD=\{K_\alpha:\alpha<\tau\}$ be a subfamily of $\scrK_K$ as in
Lemma\;\ref{lemma:decomposition}. By Lemma\;\ref{lemma:linapprox}, $\scrD$
is cofinal in $\scrK_K$\,. Now we split $\tau$ into two cofinal subsets.
For example, take\vspace{-1.5mm}
\begin{align*}
\scrC & =
\{ K_\alpha \colon \alpha<\chi(G/K), \;\alpha \text{ is an even ordinal\,} \}
\\
\scrC' & =
\{ K_\alpha \colon \alpha<\chi(G/K), \;\alpha \text{ is an \,odd ordinal\,} \}
\end{align*}\vskip -1.5mm\noindent
(Recall that an ordinal $\alpha$ is called even if $\alpha=\alpha_0+n$, where
$n$ is a finite even number and either $\alpha_0$ is a limit ordinal or
$\alpha_0=0$\,).\\
\textit{Case II:}
\,$\chi(G/K)$ is a limit cardinal. We use the following observation: \
If $\scrD$ is a family obtained by Lemma\;\ref{lemma:decomposition} and
$\tau_1$ is an infinite successor cardinal with $\chi(G/L)<\tau_1<\tau$\,,
then $\chi(G/K_{\tau_1})=\tau_1$ and (by Lemma\;\ref{lemma:linapprox})
$\scrD_{\tau_1}=\{K_\alpha:\alpha<\tau_1\}$ is cofinal in
$\scrK_{K_{\tau_1}}$. By Corollary\;\ref{corollary:inbetween}, it follows
that $\mu\conv\lambda_{K_{\tau_1}}\iin\ssMeas(G,K_{\tau_1})$. Thus
(refining Corollary\;\ref{cor:groupapprox}\,\thetag{1}\,) the groups
$L\iin\scrK_K$ for which $\mu\conv\lambda_L \iin \ssMeas(G,L)$ form a cofinal
subset of $\scrK_K$ and all infinite successor cardinals ($<\tau$) thereby
arise as $\chi(G/L)$.

Now, we split the set of cardinals $<\tau$
into two cofinal subsets.
Combined, we might take for example,\vspace{-1.5mm}
\begin{align*}
\scrC &= \{ L\iin\scrK_K \colon \chi(G/L) \text{ is an even cardinal \;and \;}
\mu\conv\lambda_L \iin \ssMeas(G,L)\,\}   \\
\scrC' &= \{ L\iin\scrK_K \colon \chi(G/L)
\text{ is an \,odd cardinal \;and \;}
\mu\conv\lambda_L \iin \ssMeas(G,L)\,\}
\end{align*}\vskip -1.5mm\noindent
(Recall that a cardinal $\tau_1=\aleph_\alpha$ is called even if $\alpha$ is
an even ordinal). Then cofinality of $\scrC,\scrC'$ easily follows.
By definition, $\ssMeast{\tau_1}(G)\perp\ssMeast{\tau_2}(G)$ \,holds for
$\tau_1\neq\tau_2$\,. This gives
\,$\mu\conv\lambda_L \perp \mu\conv\lambda_{L'}$
for all $L\iin\scrC$, $L'\iin\scrC'$.
\end{proof}

\begin{proposition}
    \label{ssnolimit}
Let $G$ be a locally compact group, $K$ a compact subgroup of $G$ such that
$G/K$ is non-metrizable. If $\mu\iin\Meas(G/K)$ but $\mu\notin\aiMeas(G,K)$,
then the net \,$(\mu\conv\lambda_L)_{L\in\,\scrK_K}$ is not convergent
for the \wstar topology of $\MeasG^{\ast\ast}$, i.e., there exists
$h\iin\MeasG^\ast$ such that
$(\,\langle\,h\,,\mu\conv\lambda_L\rangle\,)_{L\in\,\scrK_K}$ does not
converge.
\end{proposition}
    
\Proof
Decomposing $\mu$ by Theorem\,\ref{th:measureclasses}\,\thetag{iii} and
applying Lemma\,\ref{lemma:approx} to the
$\aiMeas$-com\-ponent, we may assume that
$\mu\in\ssMeas(G,K)$. Since $\mu\neq0$ there exists $f\iin C_0(G/K)$ such that
$\langle\mu,f\rangle\neq0$\,. By Corollary\;\ref{cor:groupapprox}, we have
\,$\lim_{L\in\scrK_K} \langle\mu\conv\lambda_L,f\,\rangle=\langle\mu,f\rangle$
and the same limit arises for any cofinal subfamily of $\scrK_K$\,.
Now choose subsets $\scrC$ and $\scrC'$ as in Lemma\;\ref{lemma:cofinalsets}.
By Lemma\;\ref{lemma:fextension} there exists $h\iin\MeasG^\ast$ such that
$h=f$ on the linear subspace generated by $\{\mu\conv\lambda_L\!:L\iin\scrC\}$
and $h=-f$ on the linear subspace generated by
$\{\mu\conv\lambda_L\!:L\iin\scrC'\}$ \;(with $\norm{h}\le\norm{f}$).
\end{proof}

\begin{theorem}	\label{th:caseII}
Let $G$ be a locally compact group and $K$ a compact subgroup of $G$\,.
Then \,$Z_t\bigl(\MeasG^{\ast\ast}\bigr)\cap\Meas(G/K)^{\ast\ast}\,\subseteq\;
\Meas(G/K)$.
\end{theorem}\noindent
The case of the trivial subgroup $K=\{e_G\}$ gives the Main Theorem.

\Proof
We will use induction on $\tau=\chi(G/K)$. The case where $G/K$ is metrizable
(i.e. \!$\chi(G/K)\le\aleph_0$) has been settled in Theorem \ref{th:caseI}.
Thus we can assume that $G/K$ is non-metrizable and that the theorem
holds for all subgroups $L\iin\scrK_\tau^\circ$\,.\linebreak
Put \,$M_0=\ssMeas(G,K)$, $M_1=\aiMeas(G,K)$, $M_2=\Meas(G/K)$
\,and $\widetilde M_0=\ssMeast\tau$\,,
$\widetilde M_1=\aiMeast\tau$\,,
$\widetilde M_2=\Meas_\tau$\,. Then
Theorem \ref{th:factorizationsub} and Corollary \ref{cor:thinss} show that
\[
Z_t\bigl(\MeasG^{\ast\ast}\bigr)\cap\Meas(G/K)^{\ast\ast}\subseteq
\ssMeas(G,K)\oplus\aiMeas(G,K)^{\ast\ast}.
\]
Thus (since $\MeasG\subseteq Z_t\bigl(\MeasG^{\ast\ast}\bigr)$\,) it will be
enough to consider the case where \,$\mathfrak m\in
Z_t\bigl(\MeasG^{\ast\ast}\bigr)\cap\aiMeas(G,K)^{\ast\ast}$\,.

By Lemma\,\ref{Identify}, $\Meas(G/K)=\MeasG\star \lambda_K$\,.
Under the standard embedding of the bidual
(see the beginning of Section~\ref{sec:dirsums}), it follows easily
(using \wstar density and continuity) that
\,$\Meas(G/K)^{\ast\ast}=\MeasG^{\ast\ast}\squ\lambda_K$ \,holds.
$Z_t\bigl(\MeasG^{\ast\ast}\bigr)$ being a subalgebra, the inductive
assumption implies that \,$\mathfrak m\squ\lambda_L=\mu_L\in\Meas(G/L)$
for all $L\iin\scrK_\tau^\circ$\,. Let
$\mu\iin\Meas(G/K)$ be the measure obtained by restricting $\mathfrak m$
to $C_0(G/K)$\ \;(subspace of $\Meas(G/K)^\ast)$. We get
\,$\langle \mu_L,f\rangle=\langle \mathfrak m\squ\lambda_L,f\rangle=
\langle \mu,\lambda_L\odot f\rangle=\langle \mu\conv\lambda_L, f\rangle$
\;for all\linebreak
$f\iin C_0(G/K)$. Hence \,$\mu_L=\mu\conv\lambda_L$\,. Now let
\,$\bar\delta\in\aiMeas(G,K)^{\ast\ast}$ be a
\wstar accumu\-lation point of the net \,$(\lambda_L)_{L\in\scrK_K}$\,.
By Lemma\,\ref{lemma:approx}, $(\lambda_L)_{L\in\scrK_K}$
is a right approximate unit for $\aiMeas(G,K)$. Since \,$\mathfrak m\in
Z_t\bigl(\MeasG^{\ast\ast}\bigr)$ implies
\,$\mathfrak m\in Z_t\bigl(\aiMeas(G,K)^{\ast\ast}\bigr)$, it\linebreak
follows that
\,$\mathfrak m\squ\bar\delta=\mathfrak m$ \,(see
\cite{Dales}\;Prop.\,2.9.16 and its proof). Then from\linebreak
\,$\mu\in\MeasG\subseteq Z_t\bigl(\MeasG^{\ast\ast}\bigr)$, we get, by using
an appropriate refinement of the net $(\lambda_L)_{L\in\scrK_K}$\,: \ \
$\mathfrak m\squ\bar\delta=\lim \mathfrak m\squ\lambda_L=\lim\mu_L=
\lim\mu\star\lambda_L=\mu\squ\bar\delta$\,. Hence
\,$\mathfrak m=\mu\squ\bar\delta$\,.
Since this holds for every accumulation point $\bar\delta$\,, it follows
that \,$\mathfrak m=\lim_{L\in\scrK_K}\,\mu\conv\lambda_L$ \,(\wstar limit
in~$\MeasG^{\ast\ast}$). Then Proposition\;\ref{ssnolimit} implies
\,$\mu\iin\aiMeas(G/K)$ and by Lemma\;\ref{lemma:approx} we get
\,$\mathfrak m=\mu\in\Meas(G/K)$.
\end{proof}

\begin{FRems}\vskip 3mm plus 3mm
\item{(a)} \,There are examples where $\Meas(G/K)$ is {\it not} strongly
Arens irregular (notation as in Theorem\,\ref{th:caseII}). It turns out that
for non-commutative $G$ the left topological centre $Z_t^{(1)}$ and the right
topological centre $Z_t^{(2)}$ \,(\cite{DaLa}\;Def.\,2.17) need not coincide.
One can show in a similar way
as in the proof of Theorem\,\ref{th:caseII} that
\;$Z_t^{(2)}\bigl(\Meas(G/K)^{\ast\ast}\bigr)=\Meas(G/K)$ holds for all
compact subgroups $K$ of a locally compact group $G$ \,(this can be seen as an
example for the more general approach of [left] group actions, as sketched in
Remark\,\thetag{b} after Corollary\,\ref{cor:thinss}).

Furthermore, there is the space \,$G//K$ of double cosets (see \cite{Jew},
in \cite[22.6]{Dieud}) it is denoted as $K\backslash G/K$). As
in Lemma\,\ref{Identify}, we may identify $\Meas(G//K)$ with the subalgebra
$\lambda_K\star\MeasG\star \lambda_K$ of
\,$\Meas(G,K)=\MeasG\star \lambda_K\;\subseteq\,\MeasG$ \;and it is not hard
to see that \,$Z_t^{(1)}\bigl(\Meas(G//K)^{\ast\ast}\bigr)\,=
Z_t^{(1)}\bigl(\Meas(G/K)^{\ast\ast}\bigr)\,\cap\,\Meas(G//K)^{\ast\ast}$.
Consider now \,$G=SL(2,\complex),\ K=SU(2)$ (a maximal compact subgroup).
Then one can show that $Z_t^{(1)}\bigl(\Meas(G//K)^{\ast\ast}\bigr)\supsetneq
\Meas(G//K)$ and this implies that
\[
Z_t^{(1)}\bigl(\Meas(G/K)^{\ast\ast}\bigr)\supsetneq\Meas(G/K)\ \
(=Z_t^{(2)}\bigl(\Meas(G/K)^{\ast\ast}\bigr)\,).
\]
Thus $\Meas(G//K)$ and $\Meas(G/K)$ are not strongly Arens irregular (in this
example). On the other hand, one can show that 
\,$Z_t^{(1)}\bigl(\aMeas(G//K)^{\ast\ast}\bigr)=\aMeas(G//K)$ and
$Z_t^{(1)}\bigl(\aMeas(G/K)^{\ast\ast}\bigr)=\aMeas(G/K)$.
Thus $\Meas(G//K)$ and $\Meas(G/K)$ are not Arens regular (in this
example). The convolution structure of $\Meas(G//K)$ is described explicitly
in \cite[15.5]{Jew}. In particular, one has $\mu\star\nu\in\aMeas(G//K)$ for
all $\mu,\nu\in\Meas(G//K)$ with $\mu(\{K\})=\nu(\{K\})=0$ (showing a big
difference to the abelian case). It follows that for $\mu\in\Meas(G//K)$
(with $\mu\neq0,\;\mu(\{K\})=0$) there
is no analogue of Theorem\,\ref{th:perfectset} when replacing translates
$\mu\star x$ by  ``generalized translates" $\mu\star \nu_x$ with
$\nu_x=\lambda_K\star\delta_x\star\lambda_K$\,. More details will be given
elsewhere.
\vspace{1mm plus 2mm}
\item{(b)} \,In most of the paper the focus has been on
locally compact groups and many proofs made heavy use of local compactness.
We want to sketch here another approach which allows to treat some
classes of non-locally compact groups.

First observe that if $N$ is a nowhere dense subset of a topological group
$G$\,, then (since group multiplication is an open mapping)
$\{(x,y):\,xy^{-1}\in N\}$ must be nowhere dense in $G\times G$\,.

Assume now that $G$ is a Polish group. Then one can apply a result of
Mycielski (\cite[Theorem\,1]{Myc}) and it follows that given a non-empty
meagre subset $Z$ of $G$\,, there exists
a perfect set $P$ in $G$ such that $xy^{-1}\notin Z$ for all
distinct $x$ and $y$ in~$P$. Next observe that if $G$ is {\it not} locally
compact then every compact subset must be nowhere dense. Hence, if $A$ is
any $\sigma$-compact subset of $G$\,, then $Z=A^{-1}A$ is meagre.
Applying the result above, there exists a perfect set $P$ in $G$ such that the
sets $Ax\ (x\in P)$ are pairwise disjoint. Now, if $\mu$ is an arbitrary
finite Borel measure on a non-locally compact Polish group $G$\,, then
(being a Radon measure) it is concentrated
on a $\sigma$-compact subset and it follows that $\mu$ is $2^{\aleph_0}$-thin
\;(recall that $\card{P}=2^{\aleph_0}$ for perfect subsets).

For every uncountable Polish group, $\card{\MeasG}=2^{\aleph_0}$ \
(the algebra of sets generated by a countable basis of the topology is again
countable and by elementary measure theory, a finite, $\sigma$-additive set
function on a $\sigma$-algebra is uniquely determined by its values on a
generating subalgebra), in particular
$d\bigl(\MeasG\bigr)=2^{\aleph_0}$. We can apply now the arguments of
Section \ref{sec:dirsums} and \ref{sec:factorization} \;(actually,
local compactness is not needed there) and it follows that
$Z_t\bigl(\MeasG^{\ast\ast}\bigr)=\MeasG$ holds for every Polish
group $G$. Thus $\MeasG$ is strongly Arens irregular for Polish groups.
\end{FRems}

\renewcommand{\baselinestretch}{1}\normalsize

\end{document}